\documentclass[10pt]{article}
\usepackage{color}
\usepackage{amssymb}
\usepackage{amsthm,array,amssymb,amscd,amsfonts,latexsym, url}
\usepackage{amsmath}
\usepackage[all]{xy}
\newtheorem{theo}{Theorem}[section]
\newtheorem{prop}[theo]{Proposition}
\newtheorem{claim}[theo]{Claim}
\newtheorem{lemm}[theo]{Lemma}
\newtheorem{coro}[theo]{Corollary}
\newtheorem{rema}[theo]{Remark}

\newtheorem{example}[theo]{Example}

\title{Rank $2$ vector bundles and   degrees of points of del Pezzo surfaces}
\author{Claire Voisin\footnote{The author is supported by the ERC Synergy Grant HyperK (Grant agreement No. 854361).}}
\date{}

\newfont{\gothic}{eufb10}
\begin{document}

\maketitle
\begin{abstract} We study  points and $0$-cycles on del Pezzo surfaces defined over a field $K$ of characteristic $0$, with emphasis on cubic surfaces. We prove that a cubic surface that admits a point defined over a field extension of $K$  of degree coprime to $3$ either has a $K$-point or has a point defined over a field extension of degree $4$. This improves a result of Coray (who allowed also field extensions of degree $10$). We also prove that $0$-cycles of degree $\geq 18$ on a cubic surface are effective and get similar results for degree-$2$ and degree-$1$ del Pezzo surfaces, improving results of Colliot-Th\'{e}l\`{e}ne. In a different direction, we prove that the third symmetric product of a cubic hypersurface of dimension $\geq 2$ is unirational over any field, and that, in dimension $2$  and $3$, it is not  stably rational in general.
 \end{abstract}
\section{Introduction}
Let $X\subset \mathbb{P}^n$ be a hypersurface of degree $d$ over a field $K$. When $K$ is not algebraically closed, it can happen  that not only does $X$ have no $K$-point, but all $L$-points of $X$ are defined over fields extensions $ L\supset K$ of degree divisible by $d$. The typical example is given by the generic hypersurface, where the field $K$ is the function field
$k(B)$ for some field $k$, the base $B$ being the space $\mathbb{P}(H^0(\mathbb{P}^n,\mathcal{O}_{\mathbb{P}^n}(d)))$. We take for $X$ the generic fiber $\mathcal{X}_\eta\rightarrow {\rm Spec}(K)$ of the universal family
$B\times \mathbb{P}^n\supset\mathcal{X}\stackrel{\pi}{\rightarrow} B$, where $\pi$ is given by the first projection. The $L$-points of $\mathcal{X}_\eta$ correspond,  by taking  their Zariski closure in $\mathcal{X}$, to rational multisections of $\pi$. Using the fact that $\mathcal{X}$ is a projective bundle over $\mathbb{P}^n$ via the second projection, one  sees that the restriction map  ${\rm CH}(\mathcal{X})\rightarrow {\rm CH}({X})$ has the same image as the restriction map
${\rm CH}(\mathbb{P}^n)\rightarrow {\rm CH}({X})$, which implies in particular that all $0$-cycles of $X$ are of degree divisible by $d$.

The first part of this paper is devoted to hypersurfaces of degree $3$. The above argument shows that a cubic  hypersurface over a field $K$ is not in general unirational (in particular, it is not rational) since it has no $K$-point. This argument fails however for the third punctual Hilbert scheme $X^{[3]}$  of $X$, that always contains many $K$-points, namely all  those obtained by intersecting $X$ with a line in $\mathbb{P}^n$. As a consequence of Fogarty's theorem \cite{fogarty}, the third punctual Hilbert scheme $X^{[3]}$ of a smooth variety $X$  is smooth and birational to the symmetric product  $X^{(3)}$ (see also \cite{cheah}, which proves that this statement is  true only for the second and third punctual Hilbert schemes). We will be interested in $K$-points of $X^{[3]}$ corresponding to $L$-points of $X$ for some field extension $K\subset L$ of degree $3$, that is, after restriction to an affine open set $U$ of $X$, to morphisms $\Gamma(\mathcal{O}_U)\rightarrow L$ of $K$-algebras. Thus the Hilbert scheme viewpoint is more natural than the viewpoint of the symmetric product, but in fact, these $K$-points of $X^{[3]}$ give rise to three distinct $\overline{K}$-points of $X$ permuted by the Galois group, so the Hilbert scheme is isomorphic to the symmetric product at these points and we could work as well with the symmetric product $X^{(3)}$.

We will start by proving the following easy but useful result.

\begin{theo} \label{theounirat} Let $X\subset \mathbb{P}^n_K$ be a smooth cubic hypersurface defined over a field $K$ of characteristic $0$. If $n\geq 3$, the variety $X^{[3]}$ is unirational.
\end{theo}
\begin{rema}{\rm  Our proof provides an  explicit unirational parameterization. Moritz Hartlieb and Matteo Verni observed  that one can also prove Theorem \ref{theounirat} using \cite[Proposition 37]{kollarmella}, as any smooth cubic hypersurface of dimension at least $2$  defined over a field $K$ and having  a $K$-point is unirational.}
\end{rema}
Theorem \ref{theounirat}  immediately implies the following Corollary \ref{corodense}, which answers in particular (in the negative)  the question, asked in \cite[Question 6.2]{maq} and discussed in \cite{colliot}, whether for some  cubic hypersurfaces of dimension $\geq 2$ over a field $K$, all $K$-points of $X^{[3]}$ could be obtained by intersecting a line with $X$. Given a cubic hypersurface $X\subset \mathbb{P}^n$, there is a canonical $0$-cycle \begin{eqnarray}\label{eqhtrois} h_3\in {\rm CH}_0(X)\end{eqnarray}
 defined as the intersection of a line in $\mathbb{P}^n$ with $X$. The effective $0$-cycles $\Delta\cap X$ are all rationally equivalent to $h_3$, but they are clearly not Zariski-dense in $X^{[3]}$, by a dimension count. In fact, they are characterizing by collinearity.
\begin{coro} \label{corodense} Let $X\subset \mathbb{P}^n_K$ be a smooth cubic hypersurface defined over a field $K$ of characteristic $0$. Then the $K$-points of $X^{[3]}$ are Zariski-dense in $X^{[3]}$. In particular there exist non-collinear $K$-points of $X^{[3]}$. Moreover,  the $K$-points of $X^{[3]}$  which are rationally equivalent to $h_3$ are Zariski-dense in $X^{[3]}$.
\end{coro}

In contrast, as we will prove in  Section \ref{sec1} (see Theorem \ref{propassr}), there exist  smooth cubic surfaces $S$ over a field of characteristic $0$ such that $S^{[3]}$ is not  stably rational. There also exist smooth cubic threefolds $X$ over $\mathbb{C}$ such that $X^{[3]}$ is not  stably rational. This follows from the recent results proved in \cite{SEggay}.

Coming back to the case of $K$-points of the cubic $X$ itself, Cassels and Swinnerton-Dyer conjectured that, if a cubic hypersurface $X$ has $L$-points over a field extension $ L\supset K$ of degree coprime to $3$, then $X$ has a $K$-point. Formulated in more geometric terms, if $X$ has a $0$-cycle of degree $1$, then it has a $K$-point. In \cite{coray}, Coray studied the possible  degrees coprime to $3$ of a field extension $K\subset L$ such that $X(L)$ is not empty. He proved the following
\begin{theo}(See  \cite[Theorem 7.1]{coray}\label{theocoray}.)  Let $S$ be a  smooth cubic surface over a field of characteristic $0$ having a $0$-cycle of degree $1$. Then the minimal degree coprime to $3$ of a  closed point of $S$ is  $1$, $4$ or $10$. \end{theo}

One of our main results in this paper is that we eliminate the possibility that this minimal degree  (that we will call the Coray number of $S$) is $10$.

\begin{theo} \label{theocubicsurf}  Let $S\subset \mathbb{P}^3_K$ be a smooth cubic surface defined over a field $K$ of characteristic $0$. If $S$ has a $0$-cycle of degree $1$, then $S$ either has a $K$-point or has an $L$-point, where $L$ is a field extension of $K$ of degree $4$.
\end{theo}
\begin{rema}\label{remaprevious} {\rm   That the Coray number of a cubic hypersurface cannot be $2$ follows from the fact that, if there exists an effective $0$-cycle $z$ of degree $2$ on a cubic hypersurface $X$ over a field $K$, then $X$ has a $K$-point. When $z$ comes from a length $2$ subscheme $Z_2\subset X$ which generates a line $\langle Z_2\rangle$ not contained in $X$, this follows from the construction of the residual intersection point $y$ of the line $\langle Z_2\rangle$ with $X$. When $z$ is supported on a line  contained in $X$, this is still easier.}
\end{rema}
\begin{rema} {\rm In Theorem \ref{theocubicsurf}, an equivalent conclusion is that  $S$ either has a $K$-point or has an effective $0$-cycle of degree $4$. This follows from Remark \ref{remaprevious} and from the fact that an effective $0$-cycle is an integral combination  of classes of points defined over field extensions of $K$.}
\end{rema}

\begin{rema}{\rm  In \cite{colliot}, Colliot-Th\'{e}l\`{e}ne established an analogue of Theorem \ref{theocoray} for del Pezzo surfaces $S$ of degree $2$. Namely, he proves that, if $S$ has a $0$-cycle of degree $1$, then it has a closed point of degree $1$,  $3$ or $7$. It was proved by Koll\'ar and Mella (see \cite{kollarmella}, \cite{colliot}) that the minimal odd degree of a point (the Coray number of $S$) can be $3$. It is natural to ask if one  can (as in Theorem \ref{theocubicsurf}) eliminate the possibility that the Coray number of $S$ is $7$,  but we were not able to do this. }
\end{rema}

The method applied to prove Theorem \ref{theocubicsurf} will allow us to prove more general results concerning  the effectivity of $0$-cycles  modulo rational equivalence. Such statements were first introduced in \cite{colliot}, which proved the following result
\begin{theo} (See \cite[Th\'{e}or\`{e}me 3.3]{colliot}) If a smooth cubic surface $S$ over a field $K$ of characteristic $0$  has a $K$-point $x$,  any effective  $0$-cycle $z$ of $S$ can be written as
\begin{eqnarray}\label{eqformcyclecaspoint}   z= z'+\alpha x\,\,{\rm in}\,{\rm CH}_0(S),\end{eqnarray}
where $\alpha$ is an integer and $z'=0$ or $z'$ is the class of an effective cycle  of $S$  of degree  $1$ or  $3$. In particular, ${\rm CH}_0(S)$ is generated by classes of $K$-points and points defined over field extensions of $K$ of degree $3$.
\end{theo}
Another result proved in {\it loc$.\,$cit.} concerns the effectivity of $0$-cycles of geometrically rational surfaces.
\begin{theo} \cite[Th\'{e}or\`{e}mes 3.3,  4.4, 5.1]{colliot}\label{theocollioteffective}  Let $S$ be a smooth del Pezzo surface of degree $d_S$ over a field $K$ of characteristic $0$. If $S$  has a $K$-point, then

(a) If $d_S=3$ (cubic surfaces),  any $0$-cycle of degree $d\geq 3$ on $S$ is effective.

(b) If  $d_S=2$, then any $0$-cycle of degree $d\geq 6$ on $S$ is effective.

(c) If  $d_S=1$, then any $0$-cycle of degree $d\geq 21$ on $S$ is effective.
\end{theo}
\begin{rema}{\rm In \cite{colliot}, the  bounds given   are   different  but these are  actually the bounds that can be deduced from  the results of {\it loc$.\,$cit.} modulo a  small extra argument (see the proof of Corollary \ref{coroeffdp2} below, and  also the addendum on   Colliot-Th\'{e}l\`{e}ne's website). }
\end{rema}
We will generalize  these  statements, mainly removing the assumption that $S$ has a $K$-point, and also by improving the numerical bounds given above, for example in (iii) where the assumption of having a $K$-point is automatic. Note that any del Pezzo surface $S$ of degree $d_S$  has a canonical effective $0$-cycle $$h_{d_S}=c_1(K_S)^2\in{\rm CH}_0(S),$$ which  generalizes the $0$-cycle $h_3$ of (\ref{eqhtrois}).  We will first prove the following two results for smooth cubic surfaces defined over a field of characteristic $0$.
\begin{theo} \label{theo2surfcub} (a) (Cf. Theorem \ref{theo2surfcubcasgen}) Let $S$ be a  smooth cubic surface over  a field $K$ of characteristic $0$. Then any  effective $0$-cycle $z$ of $S$ can be written as
 \begin{eqnarray}\label{eqformcycleintro} z= \pm z'+\gamma h_3\,\,{\rm in}\,\,{\rm CH}_0(S),\end{eqnarray}
 where  $z'$ is effective of degree $\leq 18$.

(b)  If $S$ has a $0$-cycle of degree $1$,  then $S$ has an effective $0$-cycle $x_4$ of degree $4$ (given by Theorem \ref{theocubicsurf}) and any effective $0$-cycle $z$   on $S$ can be written as
\begin{eqnarray}\label{eqlaformuleintropourcorayraff} z=\pm z'+\alpha x_4+\beta h_3\,\,{\rm in}\,\,{\rm CH}_0(S),
\end{eqnarray}
where $z'$ is effective of degree $\leq 4$.
\end{theo}
\begin{rema} \label{remapoursigne}{\rm  As we will see in the course of the proof of Corollary \ref{coroeff}, we can in fact impose the sign in front of $z'$ in formulas (\ref{eqformcycleintro}) and  (\ref{eqlaformuleintropourcorayraff}), that is, the four formulas  are true,  with a positive sign or negative sign.} \end{rema}

Theorem \ref{theo2surfcub} implies the following corollary that will be proved in Section \ref{secvbeff}.
\begin{coro} \label{coroeff} (a) Let $S$ be a smooth cubic surface over a field $K$ of characteristic $0$. Then any $0$-cycle $z\in {\rm CH}_0(S)$ of degree $d\geq 18$ is effective.

(b) Let $S$ be a smooth cubic surface over a field $K$ of characteristic $0$. If $S$ has a $0$-cycle of degree $1$, any
$0$-cycle $z\in {\rm CH}_0(S)$ of degree $d\geq 8$ is effective.
\end{coro}

In the case of del Pezzo surfaces of degree $2$, our results are as follows.

\begin{theo} \label{theo2surfdelPsanspoint} (a) (Cf. Theorem \ref{theogrosseborne}.)  Let   $S$ be a  smooth degree-$2$ del Pezzo surface over  a field $K$ of characteristic $0$. Then  any  effective $0$-cycle $z\in{\rm CH}_0(S)$ can be written as
 \begin{eqnarray}\label{eqformcycledp2} z= z'+\gamma h_2,
 \end{eqnarray}
 where $z'$ is effective of degree $\leq 13$.

 (b) More precisely,  any effective $0$-cycle $z\in {\rm CH}_0(S)$ can be written as
 $$z=z'+\gamma h_2\,\,{\rm in}\,\,{\rm CH}_0(S),$$
 where $z'$ is effective of degree $13, \,12$, or $\leq 7$.
 \end{theo}

 \begin{coro} \label{coroeffdp2} (a) On a smooth degree $2$ del Pezzo surface over a  field $K$ of characteristic $0$, any $0$-cycle $z\in {\rm CH}_0(S)$ of degree $d\geq 13$ is effective.

 (b) If $S$ has no  $0$-cycle of odd degree, any $0$-cycle $z\in {\rm CH}_0(S)$ of degree $d\geq 12$ is effective.
 \end{coro}
 \begin{proof} (a) Let $S$ be a del Pezzo surface of degree $2$ over a field $K$ of characteristic $0$, and let $z\in {\rm CH}_0(S)$. As $z$ is supported on a smooth curve $C\subset S$ in a linear system $|\mathcal{O}_S(m)|$ for $m$ large, we can  write by applying Riemann--Roch on $C$
  $$z=z'  -\gamma h_2\,\,{\rm in}\,\,{\rm CH}_0(S),$$ for some large integer $\gamma$, with $z'$  effective. We now apply (\ref{eqpourfindp2suppl}) to $z'$, which gives
\begin{eqnarray}\label{eqpourzdedermin} z=  z''+\gamma' h_2\,\,{\rm in}\,\,{\rm CH}_0(S),\end{eqnarray}
where $z''$ is effective of degree $\leq 13$.
If ${\rm deg}\,z\geq 13$, then $\gamma'\geq 0$ so $z$ is effective and the corollary is proved in this case.

(b) If $S$ has no   $0$-cycle of odd degree, then in the above argument, we know that ${\rm deg}\,z''\leq 12$.
It follows that, if  ${\rm deg}\,z\geq 12$,  $\gamma'\geq 0$ and $z$ is effective.
\end{proof}

Our last results concern the case of a del Pezzo surfaces of degree $d_S=1$, which will be treated by the same method in Section \ref{secdelpezzodeg1}.

\begin{theo} \label{theo2surfdelPezzodeg1}  (a)  Let   $S$ be a  smooth degree-$1$ del Pezzo surface over  a field $K$ of characteristic $0$. Then  any  effective $0$-cycle $z\in{\rm CH}_0(S)$ can be written as
 \begin{eqnarray}\label{eqformcycledp1} z=  z'+\gamma h_1,
 \end{eqnarray}
 where $z'$ is effective of degree $\leq 15$.

 (b) Under the same assumptions, any  effective $0$-cycle $z\in{\rm CH}_0(S)$ can be written as
 $$z= \pm z'+\gamma h_1,$$
 where $z'$ is effective of degree $15,\,10,\,7,\,6$ or $\leq 4$.
 \end{theo}

The following corollary is then proved in the same way as Corollary \ref{coroeffdp2}.

 \begin{coro} \label{coroeffdp1}  On a smooth del Pezzo surface of degree $1$ defined over a field $K$ of characteristic $0$, any $0$-cycle $z\in {\rm CH}_0(S)$ of degree $d\geq  15$ is effective.
 \end{coro}

 Our main tool for the proof of Theorems  \ref{theocubicsurf}, and \ref{theo2surfcub}, \ref{theo2surfdelPsanspoint} is the classical  Schwarzenberger  construction of vector bundles of rank $2$ on a smooth del Pezzo surface $S$ of degree $3$, resp. $2$, resp. $1$, starting from smooth (or {\it lci}) length-$2$ subschemes $Z$ of $S$. Sections of these vector bundles then allow us to move the cycle and a section vanishing along a cycle $h_3$, resp. $h_2$, resp. $h_1$ allows us to prove effectivity results for $Z-h_3$, resp. $Z-h_2$, resp. $Z-h_1$. This strategy is described in Section \ref{secvb} where the key Proposition \ref{prokeyprop} is proved.

 \vspace{0.5cm}

 {\bf Thanks.} {\it I thank Jean-Louis Colliot-Th\'{e}l\`{e}ne and Victor de Vries for their careful reading, questions and comments. I am also very grateful to the referee for their numerous and constructive suggestions.}

\section{Rational self-maps and the proof of Theorem \ref{theounirat}\label{sec1}}
 We will use the following  rather standard construction (see for example \cite{bogothschi}) that allows to construct rational self-maps of a cubic hypersurface and its higher symmetric powers. Let $E$ be an elliptic curve over a field $K$ and $H\in{\rm Pic}(E)$ be a line bundle of degree $d$. Then for any integer $m=sd+1$, there is a morphism
$$\mu_s:E\rightarrow E,\,x\mapsto mx-sH\in{\rm Pic}^1(E)=E.$$
Let $X\subset \mathbb{P}^n$ be a smooth cubic hypersurface over a field $K$ of characteristic $0$. For a general element $[W]\in G:={\rm Grass}(n-1, H^0(\mathbb{P}^n,\mathcal{O}_{\mathbb{P}^n}(1)))$, there is a rational map
\begin{eqnarray}\label{eqphiw}\phi_W: X\dashrightarrow \mathbb{P}^{n-2}\end{eqnarray}
given by the linear projection from the line $\Delta_W\subset \mathbb{P}^{n}$ defined by $W$. The generic fiber of $\phi_W$ is an elliptic curve over the function field of $\mathbb{P}^{n-2}$ and it carries a line bundle $H$ of degree $d=3$. The construction above thus gives for each $s$  a rational self-map
$$\mu_{s,W}: X\dashrightarrow X$$
of multiplication by $3s+1$ over $\mathbb{P}^{n-2}$. This map induces in turn for each $k$ a rational self-map
$$\mu_{s,W}^{k}: X^{(k)}\dashrightarrow X^{(k)}.$$
Using the  maps $\mu_{s,W}^{k}$, we now  construct for each $s$ a rational map defined over $K$

\begin{eqnarray}\label{eqPsi}\Psi_s: G\times G\dashrightarrow X^{[3]}.\end{eqnarray}
The construction goes as follows.
 For a generic  $[W']\in G={\rm Grass}(n-1, H^0(\mathbb{P}^n,\mathcal{O}_{\mathbb{P}^n}(1)))$, the line $\Delta_{W'}$ produces by intersection with $X$ a subscheme of length $3$ of $X$, hence a point  $\delta_{W'}$ of $X^{[3]}$.
 The rational map $\Psi_s$  is defined by
 \begin{eqnarray} \label{eqpsiexpli}
\Psi_s([W],[W'])=\mu_{s,W}^{3}(\delta_{W'}).\end{eqnarray}

As the variety $G\times G$ is rational over $K$, Theorem \ref{theounirat} will now be obtained as a consequence of the following  result.
\begin{prop}\label{prodominant}  Assume $n\geq 3$ and ${\rm char}\,K=0$. Then, for  $s=-1$, the rational map
$$\Psi_s: G\times G\dashrightarrow X^{[3]}$$
is dominant.
\end{prop}
\begin{rema}{\rm It is plausible that the statement is true for any $s\not=0$. For $s=0$, the statement does not hold because each $\mu_{s,W}$ is the identity, hence $\mu_{s,W}^3$ is also the identity, and in particular it preserves collinearity of triples on points. The image of $\Psi_s([W],[W'])$ is then  the subvariety (birational to $G$) of $X^{[3]}$ parameterizing  collinear triples of points in $X$.}
\end{rema}
\begin{rema}{\rm For $n=2$, the statement obviously does not hold since Theorem \ref{theounirat} is wrong in this case.  The conclusion of Proposition \ref{prodominant} does not hold in this case because  the rational map $\phi_W$ of (\ref{eqphiw}) is the constant map (and in particular does not depend on $W$). For any $s$,  the map $\mu_{s,W}$ is multiplication by $3s+1$ on the elliptic curve $X$, and  it has in this case the property that $\mu_{s,W}^3: X^{[3]}\dashrightarrow X^{[3]}$ preserves collinearity, as one can see by considering the Abel map of the elliptic curve $X$.}
\end{rema}
\begin{proof}[Proof of Proposition \ref{prodominant}] For $s=-1$, the morphism
$$\Phi_s: G\times X\dashrightarrow X,$$
$$\Phi_s([W],x)=\mu_{s,W}(x)$$
has a special form, namely,  on each hyperplane section $E\subset X$, the map $\mu_{E,-1}$ is the multiplication by $-2$ on the elliptic curve $E$ and  maps $x\in E$ to
$h_E-2x$, where $h_E:=c_1(\mathcal{O}_E(1))$, hence is geometrically realized by sending $x$ to the residual intersection point of the projective  line
$\mathbb{P}(T_{E,x})$ tangent to $E$ at $x$ with $X$. It follows that $\mu_{s,W}(x)$ depends only on the tangent space at $x$ of the curve $E_{W,x}$ passing through $x$. In other words, the rational map $\Phi_{-1}$ factors
  through the rational morphism
$$ G\times X\dashrightarrow\mathbb{P}(T_X)$$
which to $([W], x)$ associates the tangent line to the fiber $E_{W,x}$ passing through $x$  of the linear projection $\phi_W$.

We observe that it suffices to prove the result when  ${\rm dim}\,X=2$ since any set of three points on $X$ (or rather subscheme of length $3$) is supported on a smooth cubic surface.
Furthermore we can assume that the field is algebraically closed (for example $K=\mathbb{C}$). We choose three  general points $x,\,y,\,z$ on $X$ and have to prove that the preimage
$\Psi_{-1}^{-1}(\{x,\,y,\,z\})$ is not empty.
Looking at the construction of $\Psi_{-1}$, this preimage consists of a pencil $W$ of elliptic plane curves on $S$, and a set of three {\it collinear} points $x',\,y',\,z'$ on $X$ (generating a line $\Delta_{W'}$), such that, denoting respectively
$E_x,\,E_y,\,E_z$ the fibers of the pencil passing through $x,\,y,\,z$, we have

\begin{eqnarray}\label{eqpointcourbe} x'\in E_x,\,y'\in E_y,\,z'\in E_z,\\
\label{eqpointmultmoinsun}\mu_{E_x,-1}(x')=x,\, \mu_{E_y,-1}(y')=y,\,\mu_{E_z,-1}(z')=z.\end{eqnarray}

Using the above description of the maps $\mu_{E,-1}$, we can describe differently this fiber, namely, let
$C_x\subset X$ (resp. $C_y\subset X$, resp. $C_z\subset X$) be the curve of points $x'\in X$ (resp. $y'\in X$, resp. $z'\in X$)  such that the line
$\langle x',\,x\rangle$ (resp. $\langle y',\,y\rangle$, resp. $\langle z',\,z\rangle$) is tangent to $X$ at $x'$  resp. $y'$, resp. $z'$ (more rigorously, we should take the respective Zariski closures of these curves in $X\setminus \{x\}$,  $X\setminus \{y\}$ and $X\setminus \{z\}$). These curves, which are ramification curves of the projection of $X$ to $\mathbb{P}^2$  from $x$ (resp. $y$, resp. $z$), are well understood. They are members of the linear system  $|\mathcal{O}_X(2)|$, irreducible for general points $x,\,y,\,z$, and they contain respectively the points $x,\,y,\,z$. Furthermore, they are mobile in the sense that  any point $w$ of $X$ can be chosen not to belong to $C_x$ (resp. $C_y$, resp. $C_z$). This last point is clear since there exists a line $\Delta$ passing through $w$ and not tangent to $X$ at $w$. This line intersects $X$ at another point $x\in X$, and $w$ does not belong to $C_x$.

 For $(x',\,y',\,z')\in C_x\times C_y\times C_z$, we then choose  any elliptic plane curve $E_{x',x}$  passing through $x'$ and $x$, and similarly $E_{y',y}$  passing through $y'$ and $y$, $E_{z',z}$  passing through $z'$ and $z$.
Equations (\ref{eqpointcourbe}) and (\ref{eqpointmultmoinsun}) are then automatically satisfied since the line $\langle x',\,x\rangle$ is tangent to any such $E_x$ at $x'$ and similarly for $y$ and $z$. We need now  to impose the following  conditions~:
\begin{enumerate}
\item \label{it1} The three points $x',\,y',\,z'$ are collinear (producing the line $\Delta_{W'}$).
\item \label{it2} The three elliptic curves $E_{x',x}$, $E_{y',y}$, $E_{z',z}$ generate a pencil (producing the  pencil $W$).

\end{enumerate}
Condition \ref{it1} has to be satisfied on $C_x\times C_y\times C_z$. Given $x',\,y',\,z'$, Condition \ref{it2} has to be satisfied on
the product $\mathbb{P}^1_{x',x}\times \mathbb{P}^1_{y',y}\times \mathbb{P}^1_{z',z}$, where the line $\mathbb{P}^1_{x',x}$  parameterizes planes in $\mathbb{P}^3$ containing the line $\langle x',x\rangle$  and so on.

That the set of triples $(x',\,y',\,z')\in C_x\times C_y\times C_z$ of {\it distinct } points satisfying conditions \ref{it1} and \ref{it2} is not empty  follows now from the following
\begin{lemm} (1) For a general choice of points $x,\,y,\,z$, condition \ref{it1} is satisfied along a nonempty curve $D\subset C_x\times C_y\times C_z$ with the property that, for a general triple $(x',\,y',\,z')\in D$, we have $x\not=x',\,y\not=y',\,z\not=z'$ and the three lines  $\langle x',x\rangle$, $\langle y',y\rangle$, $\langle z',z\rangle$ are mutually non-intersecting.

(2) For a general point $(x',\,y',\,z')\in D$, condition \ref{it2} is satisfied along a curve $D'_{x',\,y',\,z'}\subset \mathbb{P}^1_{x',x}\times \mathbb{P}^1_{y',y}\times \mathbb{P}^1_{z',z}$ which is isomorphic to $\mathbb{P}^1$.
\end{lemm}
\begin{proof} (1) The set $D$ cannot contain a surface. Indeed, it would dominate otherwise one of the 3 products $C_x\times C_y$, $C_x\times C_z$, $C_y\times C_z$ under the corresponding projections. Assume it dominates $C_x\times C_y$. Then for any $(x',\,y')\in C_x\times C_y$, the third intersection point $z''$ of the line $\langle x',\,y'\rangle$ with $X$ must belong to the curve $C_z$. Using the mobility of  the curve $C_z$  with $z$ as explained above, this is not possible if $z$ is chosen generically. Next, the set $D$ has expected codimension $\leq 2$ in $C_x\times C_y\times C_z$. Indeed, on each curve $C_\bullet$, where  $\bullet=x,\,y,\,z$, we have
the inclusion
$$\mathcal{O}_{C_\bullet} (-1)\hookrightarrow W_4\otimes \mathcal{O}_{C_\bullet},$$
where $W_4=H^0(X,\mathcal{O}_X(1))^*$.
We combine these three inclusion maps to construct a morphism of vector bundles
\begin{eqnarray}\label{eqmophvnun}\phi: {\rm pr}_x^*\mathcal{O}_{C_x} (-1)\oplus {\rm pr}_y^*\mathcal{O}_{C_y} (-1)\oplus {\rm pr}_z^*\mathcal{O}_{C_z} (-1)\rightarrow W_4\otimes \mathcal{O}_{C_x\times C_y\times C_z},\end{eqnarray}
of vector bundles of respective ranks $3$ and $4$.
It is clear that the locus $D$ of collinear triples $(x',\,y',\,z')$ is contained in the locus $D_2$ where $\phi$ has rank $\leq 2$, but  we have to remove from it the sublocus $D'_2$ where $x'=y'$ or $x'=z'$, or $y'=z'$. It is a standard fact that the rank locus $D_2$ has expected codimension $\leq 2$, hence its codimension is exactly $2$ by the previous assertion. The class of $D_2$ in ${\rm CH}^2(C_x\times C_y\times C_z)$ is in fact computed following \cite{fulton}. This class is nothing but the Segre class
$s_2(\mathcal{F})$ in ${\rm CH}^2(C_x\times C_y\times C_z)$, where $\mathcal{F}:={\rm pr}_x^*\mathcal{O}_{C_x} (-1)\oplus {\rm pr}_y^*\mathcal{O}_{C_y} (-1)\oplus {\rm pr}_z^*\mathcal{O}_{C_z} (-1)$. This follows indeed from  the fact that the rank $\leq 2$ locus of $\phi$ is also the image in $C_x\times C_y\times C_z$, under the projection map $\pi: \mathbb{P}(\mathcal{F})\rightarrow C_x\times C_y\times C_z$, of the locus
$\widetilde{D}_2\subset \mathbb{P}(\mathcal{F})$ defined by the $4$ sections of $\mathcal{F}^*$ or sections of $\mathcal{O}_{\mathbb{P}(\mathcal{F})}(1)$ given by $ \phi$. Using the fact that $D_2$ has the right dimension, this is saying that
$$[D_2]=\pi_* (c_1(\mathcal{O}_{\mathbb{P}(\mathcal{F})}(1))^4)=s_2(\mathcal{F})\,\,{\rm in}\,\,{\rm CH}^2(C_x\times C_y\times C_z).$$
The degree  of $D_2$ is computed using the Whitney formula for  the total Segre class, which gives
$$ s(\mathcal{F})={\rm pr}_x^*( 1+c_1(\mathcal{O}_{C_x}(1))){\rm pr}_y^*( 1+c_1(\mathcal{O}_{C_y}(1))){\rm pr}_z^*( 1+c_1(\mathcal{O}_{C_z}(1))) \,\,{\rm in}\,\,{\rm CH}(C_x\times C_y\times C_z),$$
so that
\begin{eqnarray} \label{eqpoursegre2} s_2(\mathcal{F})={\rm pr}_x^*(c_1(\mathcal{O}_{C_x}(1))){\rm pr}_y^*( c_1(\mathcal{O}_{C_y}(1)))+{\rm pr}_x^*(c_1(\mathcal{O}_{C_x}(1))){\rm pr}_z^*(c_1(\mathcal{O}_{C_z}(1)))\\
\nonumber +{\rm pr}_y^*(c_1(\mathcal{O}_{C_y}(1))){\rm pr}_z^*(c_1(\mathcal{O}_{C_z}(1))) \,\,{\rm in}\,\,{\rm CH}(C_x\times C_y\times C_z).
\end{eqnarray}
In order  to prove the non-emptiness of $D_2\setminus D'_2$, it suffices to  show that
$${\rm deg}_{H_x}\,D_2> {\rm deg}_{H_x}\,D'_2,$$ where the degree is computed via the line bundle $H_x:={\rm pr}_x^*\mathcal{O}_{C_x}(1))$ on $C_x\times C_y\times C_z$.
The curves $C_x,\,C_y,\,C_z$ being defined by quadratic equations in $X$, we have  $${\rm deg}_{H_x}\,D'_2={\rm deg}(C_y\cdot C_z){\rm deg}_{H_x}(C_x)=12\cdot 6,$$
$${\rm deg}_{H_x}\,D_2=6\cdot 6\cdot 6,$$
from which we conclude that ${\rm deg}_{H_x}\,D_2> {\rm deg}_{H_x}\,D'_2$ and $D_2\setminus D'_2$ is non-empty.

We prove by a similar counting argument that for a general element $(x',\,y',\,z')$ of $D_2\setminus D'_2$, the three lines \begin{eqnarray}\label{eq3lines}\langle x',x\rangle, \,\langle y',y\rangle, \,\langle z',z\rangle \end{eqnarray} are mutually non-intersecting.

 (2) will  follow from the last statement. Indeed, this statement says that the  three points $\delta_{x'x},\,\delta_{y'y},\,\delta_{z'z}$ which parameterize respectively the three  lines
(ref{eq3lines})
are three points in general position in the Grassmannian of lines in $\mathbb{P}^3$. The Grassmannian $G(2,4)$ is a quadric $Q$ in $\mathbb{P}^5$, and two lines in $\mathbb{P}^3$ are non-intersecting if and only if the corresponding points in $G(2,4)=Q$ have the property that the line they generate is not contained in $Q$.   There is  thus a single orbit  under the action of ${\rm PGL}(4)$  on $G(2,4)^{[3]}$  of triples parameterizing three lines  mutually nonintersecting, as it follows from the similar statement for the action of the orthogonal group ${\rm O}(6)$ on $Q$. It thus suffices to prove the result when the three lines (\ref{eq3lines}) in $\mathbb{P}^3$  are defined  by equations
$$X_0=X_1=0,\,X_2=X_3=0,\,X_0-X_2=0,\,X_1-X_3=0.$$
An easy computation shows that the set of triples of planes
\begin{eqnarray}\label{eq3planes} p_{x',x}, \,p_{y',y}, \,p_{ z',z}\in (\mathbb{P}^3)^*
\end{eqnarray}
such that the corresponding plane $P_{x',x}$ contains $\langle x',\,x\rangle$ and so on, and such that the three linear forms $p_{x',x}, \,p_{y',y}, \,p_{ z',z}$ generate a pencil of planes (i.e. are collinear),  is a copy of $\mathbb{P}^1$ diagonally embedded in $\mathbb{P}^1_{x',x}\times \mathbb{P}^1_{y',y}\times \mathbb{P}^1_{z',z}$. Indeed,
one writes   $$p_{x',x}=u_0X_0+u_1 X_1,\,p_{y,y'}=v_2X_2+v_3X_3,$$
 $$p_{z,z'}=w(X_0-X_2)+w'(X_1-X_3),$$ for adequate choice of homogeneous coordinates on $\mathbb{P}^3$. The fact that the three planes generate a pencil then provides  equations
$$ w=au_0=-bv_2,\,w'=au_1=-bv_3
$$
for some nonzero coefficients $a,\,b$. Hence the equations provide
$$u_1/v_3=u_0/v_2=-b/a$$
$$w=au_0,\,w'=au_1,$$
which proves the result.
\end{proof}
This concludes the proof of Proposition \ref{prodominant}.
\end{proof}
We finish this section with the proof of the following result, which immediately follows from \cite{SEggay}, and is in contrast with Theorem \ref{theounirat}.
\begin{theo} \label{propassr} (a) Let $X\subset \mathbb{P}^4$ be a very general cubic threefold over $\mathbb{C}$. Then $X^{[3]}$ is not stably rational.

(b) Let $\mathcal{S}\rightarrow (\mathbb{P}^4)^*:=B$ be the universal hyperplane section of $X$, and let $S_\eta\rightarrow B_\eta$ be its generic fiber over $B$, which is a smooth cubic surface over the field $\mathbb{C}(B)$. Then $S_\eta^{[3]}$ is not stably rational.
\end{theo}
\begin{proof} Using the natural inclusion $\mathcal{S}\subset B\times X$, the relative Hilbert scheme $\mathcal{S}^{[3/B]}$ maps naturally to $X^{[3]}$ and is generically a projective bundle over $X^{[3]}$, with fiber, over a general point $z\in X^{[3]}$ parameterizing the length-$3$ subscheme $Z\subset X$, the projective line $\mathbb{P}(H^0(X,\mathcal{I}_Z(1)))$. The stable rationality of $S_\eta^{[3]}$ over $\mathbb{C}(B)$ would imply the stable rationality over $\mathbb{C}$ of $\mathcal{S}^{[3/B]}$ which is birational to a projective bundle over $X^{[3]}$ by the  construction above, and thus the stable rationality of $X^{[3]}$ over $\mathbb{C}$. We now claim that  the results of \cite{SEggay} imply that $X^{[3]}$ is not stably rational over $\mathbb{C}$, which proves both (a) and (b) by the argument above. To prove this, we recall that the stable rationality of $X^{[3]}$ implies that $X^{[3]}$ has trivial universal ${\rm CH}_0$ group so we just have to show that this is not the case.  To see this, we use the  natural incidence correspondence
$$ I\subset X^{[3]}\times X$$
defined as follows : recalling that $X^{[3]}$ is the Hilbert scheme parameterizing length-$3$ subschemes $Z\subset X$, it  carries a universal object which is described by a flat morphism $p: I\rightarrow X^{[3]}$ of degree $3$. The morphism  $I\rightarrow X$ is the inclusion as a  subscheme of $X$ of length $3$ on each fiber of $p$.
Moreover, there is  an inclusion
$$i:\mathbb{P}(\Omega_X)\hookrightarrow X^{[3]},$$
which to $x\in X,\,0\not=\eta\in \Omega_{X,x}$ associates the subscheme of length $3$ of $X$ that is supported on $x$ and has  as  Zariski tangent space at $x$ the hyperplane defined by $\eta$.
We have the obvious relation
\begin{eqnarray}\label{eqcorrespHS} I_*\circ i_*(x)=3x\,\,{\rm in}\,\,{\rm CH}_0(X).\end{eqnarray}
Assume by contradiction that $X^{[3]}$ has trivial universal ${\rm CH}_0$ group, that is, any $0$-cycle of $X^{[3]}$ of degree $0$ defined over a  field  containing $\mathbb{C}$ is trivial (the interesting case being the function field of $X^{[3]}$ itself). By the construction described above of the fat points, there is, after choosing a rational section $\sigma: X\dashrightarrow \mathbb{P}(\Omega_X)$ of the structural morphism $\mathbb{P}(\Omega_X)\rightarrow X$, a  rational map  $ i\circ \sigma: X\dashrightarrow X^{[3]}$, and thus a point $\gamma_{X^{[3]}}$ of $X^{[3]}$ over the function field $M$ of $X$. We apply the ${\rm CH}_0$-universal triviality to the difference $\gamma_{X^{[3]}}-z_0$, where $z_0=i(z'_0)$ for some  point of $X$ defined over $\mathbb{C}$.  We thus get that
\begin{eqnarray}\label{eqnewdu1702} \gamma_{X^{[3]}}-z_0=0\,\,{\rm in}\,\,{\rm CH}_0(X^{[3]}_M).
\end{eqnarray}
Applying (\ref{eqcorrespHS}), we get that
$$I_*(\gamma_{X^{[3]}})=3\delta_X \,\,{\rm in}\,\,{\rm CH}_0(X_M),$$
where $\delta_X\in X(M)$ is the generic point of  $X$.
 It thus follows   from (\ref{eqnewdu1702})  that  the cycle
$$z=\delta_X-z'_0\in{\rm CH}_0(X_M),$$
  satisfies
$$3z=0\,\,{\rm in}\,\,{\rm CH}_0(X_M).$$ On the other hand, as $X$ admits a unirational parameterization of degree $2$, we also know that $z$ satisfies $2z=0$ in ${\rm CH}_0(X_M)$. Thus $z=0$ and $X$ has trivial universal ${\rm CH}_0$ group.
This contradicts \cite{SEggay}, which proves that $X$ does not have trivial universal ${\rm CH}_0$ group (more precisely, it is proved in {\it loc$.\,$cit.} that the minimal class of the intermediate Jacobian of $X$ is not algebraic and this prevents $X$  having trivial universal ${\rm CH}_0$ group using \cite{voisinjems}).
\end{proof}

\section{Zero-cycles and rank $2$ vector bundles on del Pezzo surfaces \label{secvb}}
Let $S$ be a smooth del Pezzo surface over a field $K$ of characteristic $0$. We will denote the ample  line bundle $K_S^{-1}$ by $\mathcal{O}_S(1)$.  We denote by $d_S$ the canonical degree of $S$, namely $d_S:={\rm deg}\,c_1(K_S)^2$.
 For all $l\geq 0$, we have by Riemann--Roch :
\begin{eqnarray}\label{eqRR} h^0(S,\mathcal{O}_S(l))=1+\frac{d_S}{2}(l^2+l).\end{eqnarray}

Let $d$ be a positive integer  and let  $x\in S^{[d]}(K)$ be a  $K$-point, parameterizing a reduced (or {\it lci}) subscheme $Z_{x}\subset S$ of length $d$, which is defined over $K$. Suppose $l$ is   a nonnegative  integer    such  that
\begin{eqnarray} \label{eqpinceR} h^0(S,\mathcal{O}_S(l))<d,
\end{eqnarray}
that is, by (\ref{eqRR}),
\begin{eqnarray} \label{eqpinceplusRR} 1+\frac{d_S}{2}(l^2+l)<d.
\end{eqnarray}
 The strict inequality  in (\ref{eqpinceR}) implies that  the restriction map
$H^0(S,\mathcal{O}_S(l))\rightarrow H^0(Z_x,\mathcal{O}_{Z_x}(l))$ is not surjective, hence that
$$H^1(S,\mathcal{I}_{Z_x}(l))\not=0.$$
 As is standard (see \cite{schwarzenberger}, \cite{lazarsfeld}), we  use Serre's duality
$$H^1(S,\mathcal{I}_{Z_x}(l))^*\cong {\rm Ext}^1(\mathcal{I}_{Z_x}(l),K_S)$$
and as $K_S=\mathcal{O}_S(-1)$, it follows that the $K$-vector space
${\rm Ext}^1(\mathcal{I}_{Z_x}(l+1),\mathcal{O}_S)$ is nontrivial.
Any element $e\in {\rm Ext}^1(\mathcal{I}_{Z_x}(l+1),\mathcal{O}_S)$ provides us with a rank $2$ coherent sheaf $E$  constructed
 as  an extension
\begin{eqnarray}\label{eqextvb} 0\rightarrow \mathcal{O}_S\rightarrow E\rightarrow \mathcal{I}_{Z_x}(l+1)\rightarrow 0.
\end{eqnarray}
We now have
\begin{lemm} \label{lepourvb} Assume $x$ corresponds to a $L$-point of $S$ defined over a field extension $K\subset L$ of degree $d$. Then

(a) For any nonzero  extension class $e\in {\rm Ext}^1(\mathcal{I}_{Z_x}(l+1),\mathcal{O}_S)$, the coherent sheaf $E$ constructed from $e$  is locally free.

(b) $E$ has a section $\sigma$ whose zero-set is exactly $Z_x$.

(c)  We have
\begin{eqnarray}\label{eqdimspsec} h^0(S,E)=1+h^0(S,\mathcal{I}_{Z_x}(l+1)).
\end{eqnarray}

\end{lemm}
\begin{proof} As is well-known (see \cite[p. 726]{griffithsharris} or \cite{schwarzenberger}), the coherent sheaf $E$ is locally free  away from $Z_x$ and it is locally free  on $S$ if and only if the extension class $e$ is nonzero at each point $z$ of $Z_x$ (over the algebraic closure of $K$).   Passing to the algebraic closure of $K$, we see that the set $S_{\rm nlf}$ of points of $S_{\overline{K}}$ where $E$ is not locally free is contained in $Z_{x,\overline{K}}$ and, as  $E$ is defined over $K$,  $S_{\rm nlf}$  is invariant under ${\rm Gal}(\overline{K}/K)$. As $Z_x$ is an $L$-point of $S$, ${\rm Gal}(\overline{K}/K)$ acts transitively on the set of points in $Z_{x,\overline{K}}$. Thus $S_{\rm nlf}$ is either empty or the whole of $Z_{x,\overline{K}}$. In the second case,  the extension class $e\in {\rm Ext}^1(\mathcal{I}_{Z_x}(l+1),\mathcal{O}_S)$ vanishes in
$H^0(S,\mathcal{E}xt^1(\mathcal{I}_{Z_x}(l+1),\mathcal{O}_S))$. However, as it follows from the vanishing of $H^1(S,\mathcal{O}_S(-l-1))$, the natural map
$${\rm Ext}^1(\mathcal{I}_{Z_x}(l+1),\mathcal{O}_S)\rightarrow H^0(S,\mathcal{E}xt^1(\mathcal{I}_{Z_x}(l+1),\mathcal{O}_S))$$
is injective, so in the second case, the extension class $e$ is identically $0$, which contradicts our assumption. This proves (a).

(b)  The section $\sigma$ being given by the morphism $\mathcal{O}_S\rightarrow E$ on  the left in (\ref{eqextvb}),  (b) follows  from (a) and the exact sequence (\ref{eqextvb}).

(c) follows from the exact sequence (\ref{eqextvb}) and the vanishing $H^1(S,\mathcal{O}_S)=0$.
\end{proof}
Using rank $2$ vector bundles as in  Lemma \ref{lepourvb} will be our main tool in this paper, as they will be used to show that some $0$-cycles are effective. As a sample result, let us prove   the following statement, which will be systematically used  for  the proof of Theorems \ref{theocubicsurf} and \ref{theo2surfcub}.
\begin{prop}\label{prokeyprop}  Let $S$ be a smooth cubic surface over a field $K$ of characteristic $0$, and let $l\geq 0,\,d\geq 0$ be integers satisfying  inequality (\ref{eqpinceR}), namely
\begin{eqnarray}\label{eqineqpourvbl} h^0(S,\mathcal{O}_S(l))<d.
\end{eqnarray} Assume there is an effective cycle $z_d\in{\rm CH}_0(S)$ of degree $d$ and let $s\geq 1$ be an integer.
Then, if      $S$ has an effective $0$-cycle  $z_s$ of degree $s$, and
\begin{eqnarray}\label{eqnumersect} h^0(S,\mathcal{O}_S(l+1))-d= 1+\frac{3}{2}((l+1)^2+l+1)-d\geq 2s,
\end{eqnarray}
the cycle   $z_d-z_s\in{\rm CH}_0(S)$ is effective. In particular,
$S$ has an effective $0$-cycle $z\in{\rm CH}_0(S)$ of degree $d-s$.

\end{prop}
\begin{proof} First of all, we note that the existence of an effective $0$-cycle $z_d$ of degree $d$  defined over $K$ implies the existence of a length-$d$ subscheme $Z\subset S$ which is curvilinear, hence {\it lci}, of Chow class $z_d$. Indeed, the fibers of the Hilbert--Chow morphism
$$S^{[d]}\rightarrow S^{(d)}$$
over  $K$-points  are rational over $K$, and the subset of the fiber parameterizing curvilinear (hence {\it lci}) subschemes is open and Zariski-dense in this fiber.  The same remark applies to $z_s$.

Our assumption thus gives us a subscheme $Z_d\subset S$ of length $d$ which is {\it lci}. Using inequality (\ref{eqineqpourvbl}), we can perform the construction of the rank $2$ coherent sheaf $E$ as above. By Lemma \ref{lepourvb}, it has a section $\sigma$ which vanishes exactly along $Z_d$. Furthermore, thanks to (\ref{eqnumersect}), we get
$$h^0(S,\mathcal{I}_{Z_d}(l+1))\geq 2s,$$ hence
 $h^0(S,E)\geq 2s+1$ by Lemma \ref{lepourvb} (c). For a  subscheme $Z_s\subset S$ of length $s$, there is by (\ref{eqnumersect}) a nonzero section $\sigma'$ of $E$ vanishing along $Z_s$.
 Assuming the coherent sheaf $E$ is locally free and the zero-set  $V(\sigma')$ is of dimension $0$, then the residual cycle $Z'$ of $Z_s$ in $V(\sigma')$ is effective, defined over $K$ and of degree $d-s$; more precisely it is of class $c_2(E)-z_s=z_d-z_s$, so the proof is finished in this case.
Unfortunately, as we want to prove the result for a specific subscheme $Z\subset S$, it might be that $E$ is not locally free    and that any section of $E$ vanishing along any  subscheme $Z_s$ of length $s$ defined over $K$ vanishes along a curve in $S$ (see Example \ref{examplepasbon}).

However, we can circumvent this problem by  making both subschemes $Z_d\subset S,\,Z_s\subset S $ generic. Let $B:=S^{[d]}\times S^{[s]}$ and let $S_\eta$ be the generic fiber (defined over $K(B)$) of the projection $\pi: S\times B\rightarrow B$. Then, denoting by $[Z]\in S^{[k]}(K)$ the point parameterizing the subscheme $Z\subset S$ of length $k$ defined over $K$,  the subscheme $Z_d\subset S$  is the specialization at any point $([Z_d], [Z_s])\in (S^{[d]}\times S^{[s]}) (K)$
of   the pull-back ${\rm pr}_1^*\mathcal{Z}_d\subset  S^{[d]}\times S^{[s]}\times S$ of the universal subscheme
$$\mathcal{Z}_d\subset S^{[d]}\times S.$$
We will denote the generic fiber  of ${\rm pr}_1^*\mathcal{Z}_d$ over ${\rm Spec}\,K(B)$ by $Z_{d,\eta}\subset S_\eta$.
Using the subscheme $Z_{d,\eta}\subset S_\eta$, we now perform the construction of the  rank $2$ coherent sheaf  $E_\eta$ over $S_\eta$. Lemma \ref{lepourvb} applies in this situation, so $E_\eta$ is locally free and
thanks to (\ref{eqnumersect}), we get
$$h^0(S_\eta,\mathcal{I}_{Z_{d,\eta}}(l+1))\geq 2s,$$ hence
 $h^0(S_\eta,E_\eta)\geq 2s+1$ by Lemma \ref{lepourvb}(c).

Finally, the pull-back ${\rm pr}_2^*\mathcal{Z}_s$ to $S^{[d]}\times S^{[s]}$ of the universal subscheme
$$\mathcal{Z}_s\subset S^{[s]}\times S$$
parameterized by $S^{[s]}$ has a generic fiber $Z_{s, \eta}$ which is a subscheme of length $s$ of $S_\eta$. Using inequality (\ref{eqnumersect}), there exists a nonzero section $\sigma'$ of $E_\eta$ vanishing along the corresponding subscheme $Z_{s,\eta}\subset S_\eta$.
 Lemma \ref{lecasgeneric} proved below tells that, under our numerical assumptions,   there exists a section $\sigma'$ as above vanishing along $Z_{s,\eta}$ and with zero-locus $Z'_{d,\eta}$ of codimension $2$. It then follows that the cycle $$Z'_{d,\eta}-Z_{s,\eta}\in {\rm CH}_0(S_\eta)$$ is effective. Note that
 we have the equality of cycles
 $$V(\sigma')=Z'_{d,\eta}=c_2(E_\eta)=Z_{d,\eta}\,\,{\rm in}\,\,{\rm CH}_0(S_\eta),$$
 hence the Fulton specialization (see \cite[10.3]{fulton}) of $Z'_{d,\eta}$ to the fiber $S$ of $\pi$ over the point $([Z_d],[Z_s])\in S^{[d]}\times S^{[s]}(K)$  equals $z_d\in{\rm CH}_0(S)$. The Fulton specialization of the class of $Z_{s,\eta}$ is $z_{s}$.
The Fulton specialization of an effective cycle in $ {\rm CH}_0(S)$ is effective (see \cite[Lemme 2.10]{colliot}), hence we conclude that $z_d-z_s$ is effective, proving Proposition \ref{prokeyprop}.
\end{proof}
We reduced above the proof of Proposition  \ref{prokeyprop} to the case of the generic subscheme $Z_\eta\subset S_\eta$ of length $d$, and we can even assume without loss of generality that the field $K$ is  $\mathbb{C}$.
\begin{lemm} \label{lecasgeneric} Let $S$ be a smooth cubic surface over $\mathbb{C}$, and let $d,\,l,\,s$ be three integers satisfying the two inequalities
\begin{eqnarray}\label{eqineqassump} h^0(S,\mathcal{O}_S(l))<d,\end{eqnarray}
\begin{eqnarray}\label{eqineqassump2}1+\frac{3}{2}((l+1)^2+l+1)-d=h^0(S,\mathcal{O}_S(l+1))-d\geq 2s.
\end{eqnarray}
Then, for a general subscheme
$Z_d\subset S^{[d]}$ of length $d$, a general vector bundle $E$ constructed from an extension class $e\in {\rm Ext}^1(\mathcal{I}_{Z_d}(l+1),\mathcal{O}_S)$, and general set of $s$ points $x_1,\,\ldots,\,x_s\in S$, there exists a section $\sigma'$ of $E$ vanishing at the points $x_i$ and whose zero-set is of dimension $0$.
\end{lemm}
\begin{proof} The only part of the statement that is not proved either in Lemma \ref{lepourvb} or in the beginning of the proof of Proposition \ref{prokeyprop} is the fact that the vanishing locus of $\sigma'$ has codimension $2$. We argue by contradiction. We note first that we can assume that $Z_d$ is very general, since the considered property is Zariski open. We observe then that by very generality of $Z_d$ and countability of ${\rm Pic}(S)$, for any line bundle $M$ on $S$, we have \begin{eqnarray}\label{eqvanppoursecMI} H^0(S,\mathcal{I}_{Z_d}(M))=0\,\,{\rm if}\,\,h^0(S,M)\leq d.\end{eqnarray}

We now choose $x_1,\ldots,\,x_s$ generically and assume by contradiction that any section
$\sigma'\in H^0(S,E\otimes \mathcal{I}_{Z_s})\subset H^0(S,E)$ vanishes along a (possibly reducible) curve $C\subset S$. We can choose $C$ to be maximal with this property. We  first prove the result assuming that  the generic such curve $C$ is not disjoint from the set $\{x_1,\ldots,\,x_s\}$.
 As the set of points $\{x_1,\ldots,\,x_s\}$ is unordered, and more precisely the Galois group of the function field of the base parameterizing the data of $\{x_1,\ldots,\,x_s\},\,C$ as above acts as the full symmetric group $\mathfrak{S}_s$ on $\{x_1,\ldots,\,x_s\}$, we conclude  that the curve $C$ passes in fact  through all the points $x_i$, $i=1,\ldots s$. We can assume by countability of ${\rm Pic}(S)$ that the class of the curve $C$ does not depend on the points $x_i$, so we have $C\in |H|$ for some $H\in{\rm Pic}(S)$ independent of the points $x_i$   and $H$ satisfies
\begin{eqnarray}\label{eqstarpoursec}H^0(S,E\otimes H^{-1})\not=0.\end{eqnarray}

We now use the exact sequence (\ref{eqextvb}) and conclude from (\ref{eqstarpoursec}) that
\begin{eqnarray}\label{eqencorefinsam1309} H^0(S,\mathcal{I}_{Z_d}(l+1)\otimes H^{-1})\not=0.\end{eqnarray}
By (\ref{eqvanppoursecMI}), this implies
\begin{eqnarray}\label{eqineqpourdimh} h^0(S,\mathcal{O}_S(l+1)\otimes H^{-1})\geq d+1\end{eqnarray}

However,  we also have \begin{eqnarray}\label{eqineqpourdimhaumoins} h^0(S, H)\geq s+1,\end{eqnarray} since there exists a member $C$ of $|H|$ passing through a general set of $s$ points in $S$.
This provides  us with a contradiction for $s\geq 2$. Indeed, if $s\geq 2$,  we conclude from (\ref{eqineqpourdimhaumoins}) that the degree of the curves $C$ in $|H|$ is at least $3$ and
this contradicts  the following
\begin{claim} \label{claim} For any curve $D\subset S$ such that $h^0(S,E(-D))\not=0$, we have ${\rm deg}\,D\leq 2$.
\end{claim}
\begin{proof} Indeed, if ${\rm deg}\,D\geq 3$, the rank of the restriction map
$$H^0(S,\mathcal{O}_S(l+1))\rightarrow H^0(D,\mathcal{O}_D(l+1))$$ is at least $3(l+1)$. Hence
$$h^0(S,\mathcal{O}_S(l+1))\geq h^0(S,\mathcal{O}_S(l+1)(-D))+3(l+1).$$
As $h^0(S,E(-D))\not=0$,  $h^0(S,\mathcal{I}_{Z_d}(l+1)(-D))\not=0$ by the exact sequence (\ref{eqextvb}). By (\ref{eqvanppoursecMI}), we thus conclude that
$$h^0(S,\mathcal{O}_{S}(l+1)(-D))> d \geq h^0(S,\mathcal{O}_{S}(l))+1.$$
 This provides a contradiction since $h^0(S,\mathcal{O}_S(l+1))-h^0(S,\mathcal{O}_S(l))=3(l+1)$.
\end{proof}

We next deal with the case $s=1$. By the case $s\geq 2$ which is already treated, we can assume that \begin{eqnarray}\label{eqasshO} h^0(S,\mathcal{O}_S(l+1))=d+2\,\,{\rm or} \,\,h^0(S,\mathcal{O}_S(l+1))=d+3.
\end{eqnarray}
In this case, the inequality $$ h^0(S,\mathcal{O}_S(l+1)\otimes H^{-1})\geq d+1$$
of (\ref{eqineqpourdimh}) gives us
$$h^0(S,\mathcal{O}_S(l+1)\otimes H^{-1})\geq h^0(S,\mathcal{O}_S(l+1))-2,$$
which is impossible since $l\geq 0$, hence the curve $C$ which is mobile  imposes at least $3$ conditions on the linear system $|\mathcal{O}_S(1)|$.

In order to conclude the proof, we need to study the case where the curve $C$ is disjoint from $\{x_1,\ldots,\,x_s\}$.
By assumption, there exists for a generic set $Z_s$ of points $x_i$, $i=1,\ldots,\,s$, a curve $C$ and a section $\sigma'$ of $E(-C)$ which vanishes at all the points $x_i$, while the curve $C$ does not pass through any of them.   It follows that
$$\sigma'\in H^0(S,E(-C)\otimes \mathcal{I}_{Z_s}).$$
However, we have
\begin{eqnarray}\label{eqmarre1}h^0(S,E(-C))\leq h^0(S,\mathcal{I}_{Z_d}(l+1)(-C)).\end{eqnarray}
We can obviously assume that
\begin{eqnarray}\label{eqmarre2}h^0(S,\mathcal{I}_{Z_d}(l+1))=2s\,\,{\rm or}\,\,h^0(S,\mathcal{I}_{Z_d}(l+1))=2s+1.\end{eqnarray}
 It follows from (\ref{eqmarre1}),  (\ref{eqmarre2}) and (\ref{eqvanppoursecMI}) that
 \begin{eqnarray}\label{eqmarre3} h^0(S,E(-C))\leq 2s+1-(l+2) <2s, \end{eqnarray}
 because $C$ imposes at least $l+2$ conditions on $H^0(S,\mathcal{O}_S(l+1))$, hence also on $H^0(S,\mathcal{I}_{Z_d}(l+1))$.
 We know that there exists a nonzero section of $E(-C)$ vanishing  along a generic set of $s$ points in $S$.
  We claim that this  implies that for a generic set of $s$ points in $S$, there exists a curve $C'$ passing through at least one of these points and such that any section of $E(-C)$ vanishing  at these $s$ points  vanishes along $C'$. Indeed, this is done by a dimension count: we consider the universal vanishing locus
 $$\Gamma\subset \mathbb{P}(H^0(S,E(-C)))\times S^{[s]}$$
 and its Zariski open set $\Gamma_f\subset \Gamma$ consisting of pairs $(\sigma,\,\{x_1,\ldots,\,x_{s}\})$ where the $x_i$ are all distinct and the  $0$-locus of $\sigma$ is $0$-dimensional near all $x_i$.  Using  (\ref{eqmarre3}), we have ${\rm dim}\,\Gamma_f\leq 2s-1$, hence $\Gamma_f$ cannot dominate $S^{[s]}$ by the second projection, which proves the claim.

  By the claim, we are now reduced to the previous situation where the curve of vanishing of the section $\sigma'$  contains at least one of the points $x_i$. The lemma is thus fully proved.
 \end{proof}
 Let us give an example illustrating the difficulty  for a nongeneric choice of $Z\subset S$.

\begin{example} \label{examplepasbon} {\rm  Assume that $S$ contains a line $\Delta$ with residual conic $C\in|\mathcal{O}_S(1)(-\Delta)|$. Consider the vector bundle $E=\mathcal{O}_S(\Delta)\oplus \mathcal{O}_S(C)$ on $S$.  If we take a general section $\sigma$, its zero-locus $Z$ is the intersection $\Delta\cap C$ of the line $\Delta\subset S$ with one conic in the pencil, hence consists of $2$ points, and we have an extension
$$0\rightarrow \mathcal{O}_S\rightarrow E\rightarrow \mathcal{I}_Z(1)\rightarrow 0.$$
The vector bundle thus  has $3$ sections, and a general section has a length-$2$ vanishing locus,  but a section vanishing at a general point vanishes along a conic.}
 \end{example}

\subsection{Another  useful effectivity result \label{secvbeff}}
We will use in the next sections  the following result for $0$-cycles on a surface. A similar statement already appears in \cite{colliot} where it is used to simplify Coray's proof of Theorem \ref{theocoray}.
\begin{lemm}\label{leeffective} Let $S$ be a smooth projective surface over a field $K$ of characteristic $0$ and let $H$ be a very ample line bundle on $S$. Then if $Z\subset S$ is a subscheme of length $d$, of class $z\in{\rm CH}_0(S)$, and
\begin{eqnarray}\label{eqineq} d\leq h^0(S,H)-2,
\end{eqnarray}
the $0$-cycle $c_1(H)^2-z\in {\rm CH}_0(S)$ defined over $K$ is rationally equivalent to an effective $0$-cycle on $S$.
\end{lemm}

\begin{proof} Consider the smooth projective variety $B=S^{[d]}$, which is defined over $K$, with function field
$K(B)$. Let $\eta={\rm Spec}\,K(B)$ be its generic point. The universal subscheme
$$\mathcal{Z}_d\subset B\times S$$
has for generic fiber a subscheme
$${Z}_\eta\subset S_\eta$$
of length $d$, which specializes to $Z\subset S$ at the point $[Z]\in S^{[d]}(K)$ parameterizing $Z$.

\begin{claim} \label{claim1802} The effectivity statement of Lemma \ref{leeffective} is true for the generic subscheme ${Z}_\eta $.
\end{claim}
\begin{proof} Indeed, we have $h^0(S_\eta,\mathcal{I}_{Z_\eta}(H))\geq 2$ by (\ref{eqineq}). It thus suffices to show that there are two curves $C_1,\,C_2$ in
$|\mathcal{I}_{Z_\eta}(H)|$ which intersect properly, since then $C_1\cap C_2$ is a $0$-dimensional subscheme containing $Z$, so the class of $C_1\cap C_2- Z$ is effective. The existence of $C_1,\,C_2$ as above follows once the linear system  $|\mathcal{I}_{Z_\eta}(H)|$ has no fixed component. Assume by contradiction there is a fixed component  $C$. We discuss as before the two cases where  $C$ intersects $Z_\eta$ or is disjoint from $Z_\eta$.
In the first case,  $C$ contains $Z_\eta$ by a symmetric group argument, and so we conclude by (\ref{eqvanppoursecMI}) that $h^0(S,\mathcal{O}_S(C))\geq d+1$. As we can obviously assume that $d= h^0(S,H)-2$, and we have $h^0(S,H(-C))\geq 2$, we get a contradiction in this case, since $H$ is very ample.

In the second case, the curve $C$ does not intersect $Z_\eta$ so we get that $h^0(S,H(-C)\otimes \mathcal{I}_{Z_\eta}))\not=0$. This implies by (\ref{eqvanppoursecMI}) that $h^0(S,H(-C))\geq d+1$, and, as $h^0(S,H) =d+2$ and $H$ is very ample, this also provides a contradiction, which proves the claim.
 \end{proof}
The result then follows from  Claim \ref{claim1802}, using  Fulton's specialization   from ${\rm CH}_0(S_\eta)$ to ${\rm CH}_0(S)$ at the point $[Z]\in B(K)$, which, as noted in \cite{colliot}, preserves effectivity.
\end{proof}
For the applications, we will need the following variant, which is proved exactly in the same way.
\begin{lemm}\label{levariant} Let $S$ be a smooth projective surface and $L=\mathcal{O}_S(1)$ a line bundle on $S$, which is assumed to be  ample. Let $d$,  $l$ be integers such that  $\mathcal{O}_S(l+1)$ is generated by its sections and
\begin{eqnarray}\label{eq1levar} h^0(S,\mathcal{O}_S(l))\geq d+1,\\
\label{eq1levar2} h^0(S,\mathcal{O}_S(l+1))-d\geq 2.\end{eqnarray}
Then if $z_d\in {\rm CH}_0(S)$ is effective of degree $d$, the cycle
$$z':=l(l+1) c_1(L)^2-z_d\in {\rm CH}_0(S)$$
is effective.
\end{lemm}
\begin{proof} The proof works as before. The inequality (\ref{eq1levar}) guarantees that the generic subscheme $Z_{d,\eta}$ of $S$ of length $d$ is supported on a curve $C_1$ which is a member of $|\mathcal{O}_{S_\eta}(l) |$. The inequality (\ref{eq1levar2})  guarantees by the proof of  the previous lemma that the linear system  $|\mathcal{I}_{Z_{d,\eta}}(l+1) |$ has no fixed component, hence has a member $C_2$ which intersects $C_1$ properly.  Hence the generic subscheme
$Z_{d,\eta}$ is supported on the complete intersection of the curves $C_1$ and $C_2$, so the cycle
$$l(l+1) c_1(L)^2-Z_{d,\eta}\in {\rm CH}_0(S_\eta)$$
is effective.  Lemma \ref{levariant} then  follows by Fulton's specialization.
\end{proof}

 Lemmas \ref{leeffective} and \ref{levariant} will be used in next  section to deduce Corollary \ref{coroeff} from Theorem \ref{theo2surfcub}.

\section{Zero-cycles on smooth cubic surfaces \label{seccubsurf}}
We prove in this section Theorems \ref{theocubicsurf} and \ref{theo2surfcub}, and Corollary \ref{coroeff}. We start with the proof of  Theorem \ref{theo2surfcub}(a), which is  the following statement.
\begin{theo} \label{theo2surfcubcasgen} Let $S$ be a  smooth cubic surface over  a field $K$ of characteristic $0$. Then any  effective $0$-cycle $z$ of $S$ can be written as
 \begin{eqnarray}\label{eqformcycle} z=\pm z'+\gamma h_3\,\,{\rm in}\,\,{\rm CH}_0(S),\end{eqnarray}
 where $z'$ is effective of degree $\leq 18$.
\end{theo}
 Let us first prove
\begin{prop}\label{prodescent} Let $S$ be a smooth  cubic surface over a field $K$ of characteristic $0$. Let $z\in{\rm CH}_0(S)$ be effective  of degree $d$. Then, if $d\geq 20$, the cycle
$z$ can be written as
$$z=\pm  z'+\gamma h_3\,\,{\rm in}\,\,{\rm CH}_0(S),$$ where  $\gamma$ is an integer and $z'$ is effective of degree $<20$.\end{prop}
\begin{proof} We argue by induction on $d\geq 20$. Given an effective $0$-cycle $z\in{\rm CH}_0(S)$ of degree $d$, we introduce the non-negative integer $l$ such that
\begin{eqnarray}\label{eqineqdl}  h^0(S,\mathcal{O}_S(l))< d\leq  h^0(S,\mathcal{O}_S(l+1))\end{eqnarray}
We first assume that $d\leq  h^0(S,\mathcal{O}_S(l+1))-2$. By Lemma \ref{leeffective}, the cycle
$(l+1)^2h_3-z$ is effective, hence if
${\rm deg}\,((l+1)^2h_3-z)< {\rm deg}\,z$, we can replace $z$ by the effective cycle  $$z'=(l+1)^2h_3-z,$$ to which we can apply the inductive argument. In conclusion, if $d\leq  h^0(S,\mathcal{O}_S(l+1))-2$, we can assume that
\begin{eqnarray}\label{eqineqmoitie} {\rm deg}\,z\leq \frac{1}{2} {\rm deg}\,((l+1)^2h_3)=\frac{3}{2} (l+1)^2.\end{eqnarray}
By the strict  left inequality in (\ref{eqineqdl}), we can apply the
 vector bundle construction  of Section \ref{secvb}. Using (\ref{eqineqmoitie}), we get that
$$h^0(S,\mathcal{O}_S(l+1))-{\rm deg}\,z\geq 1+\frac{3(l+1)}{2}.$$
Thus Proposition \ref{prokeyprop}(b) (with $s=3$) and  Corollary  \ref{corodense} tell that $z-h_3$ is effective if  $$1+\frac{3(l+1)}{2}\geq 6,$$ that is, $l\geq 3$. In conclusion, we  proved the induction step when
  $l\geq 3$, hence when $d\geq 20=h^0(S,\mathcal{O}_S(3))+1$, assuming   that we do not have
\begin{eqnarray}\label{eqmauvaiscas} d=  h^0(S,\mathcal{O}_S(l+1)),\,\,{\rm  or} \,\,d=h^0(S,\mathcal{O}_S(l+1))-1. \end{eqnarray}

The missing  cases (\ref{eqmauvaiscas}) are now  treated as follows.
We consider, instead of the effective  cycle $z$, the  effective cycle $z'=z+h_3$. Then, in both cases, we have \begin{eqnarray}\label{eqpourpropext} {\rm deg}\,z'>h^0(S,\mathcal{O}_S(l+1)).
\end{eqnarray} We also have
\begin{eqnarray}\label{eqineqpourargfin} h^0(S,\mathcal{O}_{S}(l+2))-{\rm deg}\,z'=h^0(S,\mathcal{O}_S(l+2))-{\rm deg}\,z-3.\end{eqnarray}
As ${\rm deg}\,z\leq  h^0(S,\mathcal{O}_S(l+1))$, (\ref{eqineqpourargfin})  gives
$$ h^0(S,\mathcal{O}_{S}(l+2))-{\rm deg}\,z'\geq 3(l+2)-3,
$$
hence, if $l\geq 3$, we get
 \begin{eqnarray}\label{eqineqpourcaspenible}h^0(S,\mathcal{O}_{S}(l+2))-{\rm deg}\,z'\geq 12.\end{eqnarray}
Proposition \ref{prokeyprop} thus applies with $s=6$ once $l\geq 3$ and this says that the cycle $$z'-2h_3=z+h_3-2h_3=z-h_3\in{\rm CH}_0(S)$$
is effective. This  concludes the proof  of the induction step in the case (\ref{eqmauvaiscas}), since for $l\leq 2$, (\ref{eqmauvaiscas}) implies $d\leq 19$.
\end{proof}
We now complete the proof of Theorem \ref{theo2surfcubcasgen}.
\begin{proof}[Proof of Theorem  \ref{theo2surfcubcasgen}]
Using Proposition \ref{prodescent}, we get that any effective $0$-cycle of degree $>19$ on $S$ can be written
as $$z=z'\pm \lambda h_3\,\,{\rm in}\,\,{\rm CH}_0(S),$$
where $z'$ is effective of degree $\leq 19$. To complete the proof of Theorem \ref{theo2surfcubcasgen}, it thus suffices to  study the case  of an effective cycle $z$ of degree $d= 19$.
  Let $z':=z+3h_3\in {\rm CH}_0(S)$. This is an effective $0$-cycle of degree $28=h^0(S,\mathcal{O}_S(4))-3$.  It follows from Lemma \ref{leeffective} that the cycle $$z'':= 16h_3-z'=13 h_3-z$$ is effective of degree $20=h^0(S,\mathcal{O}_S(3))+1$.
We now apply  Proposition \ref{prodescent}  to $z''$, and we conclude  that the cycle
$$z''-h_3=12h_3-z$$ is effective of degree $17$ on $S$. Theorem \ref{theo2surfcubcasgen} is proved.
\end{proof}

We will  next  prove   the following result.
\begin{prop} \label{theonouvellecoray} (a)  Let $S$ be a smooth cubic surface over  a field $K$ of characteristic $0$ which has an effective $0$-cycle of degree  $\leq 17$ coprime to $3$. Then $S$ has a point of degree $1$ or $4$.

(b)  Under the same assumptions, denoting by $x_4$ the class of an effective cycle of degree $4$, any effective $0$-cycle $z$ of degree $\leq 17$ on $S$ can be written as
\begin{eqnarray}\label{eqpour0cyclecascoray} z=\pm z'+\alpha x_4+\beta h_3\,\,{\rm in}\,\,{\rm CH}_0(S),
\end{eqnarray}
where $z'$ is effective of degree $\leq 4$.
\end{prop}

Before proving Proposition \ref{theonouvellecoray}, let us explain how it implies both Theorems \ref{theocubicsurf} and Theorems \ref{theo2surfcub}(b).
\begin{proof}[Proof of Theorems \ref{theocubicsurf} and  \ref{theo2surfcub}(b)] Let $S$ be a smooth cubic surface over a field $K$ of characteristic $0$, which has an effective $0$-cycle of degree   coprime to $3$.  By Theorem \ref{theo2surfcubcasgen}, any effective $0$-cycle $z$ can be written as in (\ref{eqformcycle}), that is,
\begin{eqnarray}\label{eqlaformule} z=\pm z'+\gamma h_3\,\,{\rm in}\,\,{\rm CH}_0(S),\end{eqnarray}
 where $z'$ is effective of degree $\leq 18$. By assumption, $S$ has an effective $0$-cycle $z$ of degree coprime to $3$, for which  the cycle $z'$ in  (\ref{eqlaformule}) has degree $\leq 17$. By Proposition \ref{theonouvellecoray}(a), we then conclude that $S$ has a $K$-point or a point $x_4$ of degree $4$, which proves Theorem \ref{theocubicsurf}.

We next prove  Theorem \ref{theo2surfcub}(b). Starting from any effective $0$-cycle $z\in {\rm CH}_0(S)$, we write the decomposition (\ref{eqlaformule}). If
 ${\rm deg}\,z'\leq 17$, then we apply Proposition \ref{theonouvellecoray}(b) to $z'$ and thus we get a decomposition for $z'$  which takes the form
 \begin{eqnarray}\label{eqpour0cyclecascorayprime} z'=\pm z''+\alpha x_4+\beta h_3\,\,{\rm in}\,\,{\rm CH}_0(S),
\end{eqnarray}
where $z''$ is effective of degree $\leq 4$. Combining (\ref{eqpour0cyclecascorayprime}) with (\ref{eqlaformule}), we get  formula (\ref{eqlaformuleintropourcorayraff}) for $z$, so the conclusion of Theorem \ref{theo2surfcub}(b) is proved in this case.

In order to conclude the proof, it only remains to study the case where ${\rm deg}\,z'=18$. In this case, we denote by $x_4$  the class of an effective $0$-cycle of degree $4$ that exists on $S$ by Theorem \ref{theocubicsurf} which is now proved.
We observe that the cycle
$$z''=z'+2x_4+h_3$$
is effective of degree $29=h^0(S,\mathcal{O}_S(4))-2$. Thus by Lemma \ref{leeffective}, the cycle
$$z'''=16h_3-z''=15h_3-2x_4-z'$$
is effective of degree $48-29=19$. If we now look at the proof of Theorem  \ref{theo2surfcubcasgen}, we see that it proves more precisely that effective cycles $w$ of degree $19$ on $S$ can be written as
$$w=-w'+12 h_3\,\,{\rm in}\,\,{\rm CH}_0(S),$$
where $w'$ is effective of degree $17$. We can then  apply Proposition \ref{theonouvellecoray}(b) to $w=z'''$, and get formula (\ref{eqlaformuleintropourcorayraff}) for $z$ also in this case, so Theorem  \ref{theo2surfcub}(b) is fully proved.
\end{proof}

We now prove Proposition \ref{theonouvellecoray}.
\begin{proof}[Proof of Proposition \ref{theonouvellecoray}(a)] Given a smooth cubic surface $S$ having a $0$-cycle of degree $1$ over a field $K$ of characteristic $0$,  we know by Coray's Theorem \ref{theocoray} that $S$ has a point of degree $1,\,4$ or $10$, so we only have to study effective $0$-cycles  of degree $10$ on $S$.
Let $z\in {\rm CH}_0(S)$ be such a cycle. Then, if $z'=z+2h_3$, we have ${\rm deg}\,z'=16=h^0(S,\mathcal{O}_S(3))-3$, hence by Lemma \ref{leeffective}, the cycle
$$z'':=9h_3-z'=7h_3-z$$
is effective of degree $11=h^0(S,\mathcal{O}_S(2))+1$. We can thus apply the vector bundle method of Section \ref{secvb} with $l=2$.
As furthermore we have $h^0(S,\mathcal{O}_S(3))-{\rm deg}\,z'=8$, we can apply Proposition \ref{prokeyprop} with $s=3$ and conclude that
the cycle $z''-h_3$ is effective. As ${\rm deg}\,z''=8=h^0(S,\mathcal{O}_S(2))-2$, we get  an effective $0$-cycle of degree $4$ by applying Lemma \ref{leeffective}.
\end{proof}

\begin{proof}[Proof of Proposition \ref{theonouvellecoray}(b)] Let $z\in {\rm CH}_0(S)$ be effective of degree $d\leq 17$. We observe that $h^0(S,\mathcal{O}_S(3))\geq {\rm deg}\,z+2$, so by Lemma \ref{leeffective},
$$z'=9h_3-z$$
is effective. As ${\rm deg}\,z'=27-{\rm deg}\,z$,
we conclude that it suffices to prove the statement for effective $0$-cycles of degree $\leq 13$.
If $z$ has degree $d$ with $11\leq d\leq 13$, we have $d>h^0(S,\mathcal{O}_S(2))$, so the vector bundle strategy of Section \ref{secvb} works with $l=2$, and as we have $h^0(S,\mathcal{O}_S(3))=19\geq d+6$, Proposition \ref{prokeyprop} applies with $s=3$ and shows that $z-h_3$ is effective. We are thus reduced to the case of a $0$-cycle  $z$ of degree $\leq 10$. The case of degree $d=10$ has been treated in the course of the proof of Theorem \ref{theonouvellecoray}(a), so we only have to treat $d\leq 9$.
When $z$ is effective of degree $d=9$, we consider the cycle $z'=z+2x_4$ which has degree $17=h^0(S,\mathcal{O}_S(3))-2$. By Lemma \ref{leeffective}, the cycle $z''=9h_3-z'$ is effective and has degree $10$, and we proved in the proof of Theorem \ref{theonouvellecoray}(a) that any effective $0$-cycle $z''$ of degree $10$   can be written as a sum
\begin{eqnarray}\label{eqquonaprouveenfait} z''=w+2h_3\,\,{\rm in}\,\,{\rm CH}_0(S), \end{eqnarray}
 where $w$ is effective of degree $4$.

 It thus remains to study effective $0$-cycles $z$ of degree $d\leq 8$. When $6\leq d\leq 8$, the cycle
 $4h_3-z$ is effective of degree $\leq 6$ by Lemma \ref{leeffective} so we are reduced to the case of an effective cycle $z$ of degree $\leq 6$.
 In the case where $d=5$, then $z+h_3$ has degree $8=h^0(S,\mathcal{O}_S(2))-2$ so $4h_3-(z+h_3)$ is effective of degree $4$, proving the result. Finally, when $d=6$, $z+x_4$ has degree $10$, and we can use again (\ref{eqquonaprouveenfait}). The proof is thus finished.
 \end{proof}

\begin{proof}[Proof of Corollary \ref{coroeff}]  (a)  Let $S$ be a smooth cubic surface over a field $K$ of characteristic $0$. First of all, we prove that, as mentioned  in Remark \ref{remapoursigne},  we can in fact impose in formula (\ref{eqlaformuleintropourcorayraff}) the sign $\pm$ to be positive (or negative). Indeed, we claim that, if a $0$-cycle  $z$ is effective of degree $\leq 18$, then we can write
$$z=\lambda h_3-z',$$
where $z'$ is effective of degree $\leq 18$. To see this, suppose first ${\rm deg}\,z=18$. As $h^0(S,\mathcal{O}_S(3))=19$, we get by  Lemma \ref{levariant} that   the cycle $z':=12 h_3-z$, which is also  of  degree $18$, is effective, so the claim is proved in this case. Next, if $z'$ is an effective $0$-cycle of degree $17$, Lemma \ref{leeffective} implies that $z':=9h_3-z$ is effective. Furthermore, if ${\rm deg}\,z\geq 9$, ${\rm deg}\,z'\leq 18$. Finally, if ${\rm deg}\,z\leq 8$, then $4h_3-z$ is effective of degree $\leq 12$, again by application of Lemma \ref{leeffective}. So the claim is proved. We then deduce from Theorem \ref{theo2surfcubcasgen} that any effective cycle
$z\in{\rm CH}_0(S)$ can be written as
\begin{eqnarray}\label{eqmemenouvellefin} z=z'+\alpha h_3\,\,{\rm in}\,\,{\rm CH}_0(S)\end{eqnarray}
for some integer $\alpha$, where $z'$ is effective of degree $\leq 18$.
Finally, (\ref{eqmemenouvellefin}) holds as well for any $0$-cycle on $S$, effective or not, since an easy argument involving Riemann--Roch on curves in $S$ shows that for $z\in{\rm CH}_0(S)$, there exists a $\gamma\in\mathbb{Z}$ such that $z+\gamma h_3$ is effective.
Let now  $z\in{\rm CH}_0(S)$ be a $0$-cycle with ${\rm deg}\,z\geq 18$. We write $z$ as in (\ref{eqmemenouvellefin})
where $z'$ is effective of degree $\leq 18$.
As ${\rm deg}\,z\geq 18$ and ${\rm deg}\,z'\leq 18$, we have $\gamma \geq 0$, hence $\gamma h_3$ is effective and $z$ is effective.

\vspace{0.5cm}

(b) Let $S$  be as above and admitting a $0$-cycle of degree $1$. Let $z\in {\rm CH}_0(S)$ be a $0$-cycle of degree $\geq 8$. We first prove that in Theorem \ref{theo2surfcub},  the sign in front of $z'$ can be chosen positive, that is, we can write
\begin{eqnarray} \label{eqencoreettoujours}z= z''+\gamma h_3 +\delta x_4\,\,{\rm in}\,\,{\rm CH}_0(S),\end{eqnarray}
where $z''$ is  effective of degree $\leq 4$.
To see this, we start from formula (\ref{eqlaformuleintropourcorayraff}) and notice that the cycles $z''$ and $x_4$ can be (at least generically as in the proof of Lemma \ref{leeffective}) chosen supported on a smooth curve $C\in |\mathcal{O}_S(2)|$, since $h^0(S,\mathcal{O}_S(2))=10$. As $g(C)=4$, the degree-$4$ $0$-cycle $z'':=4h_3-x_4-z'$ supported on $C$ is effective by Riemann--Roch, so we can replace $z'$ by $-z''$ in formula (\ref{eqlaformuleintropourcorayraff}), getting (\ref{eqencoreettoujours}) if the original sign was negative.

Assume now that  ${\rm deg}\,z\geq 8$. Then, as ${\rm deg}\,z''\leq 4$, we have ${\rm deg}\,(\gamma h_3+\delta x_4)\geq 4$. Any $0$-cycle of the form $\gamma h_3+\delta x_4$ which is of degree $\geq 4$ is effective, for the same reason as before, since we can arrange (at least for a generic choice of $x_4$) the class $h_3$ and $x_4$ to be supported on a smooth curve $C\subset S$ of genus $4$.
This implies by Riemann-Roch  that  $z$ is effective for a generic choice of $x_4$. Using as before the specialization  from the generic case, we get that any cycle $z$ as above is effective.
\end{proof}

\section{Zero-cycles on degree-$2$ del Pezzo surfaces \label{secdelpezzo}}
We prove in this section Theorem  \ref{theo2surfdelPsanspoint} concerning degree-$2$ del Pezzo surfaces. As in the cubic surface case, the key tool is the rank $2$ vector bundle construction from Section \ref{secvb}. In that section however, Proposition \ref{prokeyprop} had been fully established only in the cubic surface case. Let us first prove the analogous result for degree-$2$ del Pezzo surface.
\begin{prop}\label{prokeypropdp2} Let $S$ be a del Pezzo surface of degree $2$ over a field of characteristic $0$. Let $d,\,l,\,s$ be three nonnegative integers with $l\geq 1$, and let $z_d\in{\rm CH}_0(S),\,z_s\in{\rm CH}_0(S)$ be two effective $0$-cycles on $S$ of respective degrees $d$ and $s$.
Assume the following inequalities are satisfied
\begin{eqnarray}\label{eqineqdp21} h^0(S,\mathcal{O}_S(l))=1+l^2+l<d,
\end{eqnarray}
\begin{eqnarray}\label{eqineqdp22}  h^0(S,\mathcal{O}_S(l+1))-d=1+(l+1)^2+l+1-d\geq 2s.
\end{eqnarray}
Then the cycle $z_d-z_s$ is effective.
\end{prop}

\begin{proof} The assumptions are exactly the same as in Proposition \ref{prokeyprop} (only the Riemann--Roch formula has changed). The strategy of the proof is of course the same and the only point that needs to be checked more precisely is the analog of Lemma \ref{lecasgeneric}, the proof of which was specific to the  cubic surface case. In other words, we reduced the proof to showing
\begin{lemm} \label{lecasgenericdp2} Let $S$ be a smooth degree-$2$ del Pezzo  surface over $\mathbb{C}$, and let $d,\,l\geq 1,\,s$ be three nonnegative integers satisfying the two inequalities
\begin{eqnarray}\label{eqineqassumpdp2} h^0(S,\mathcal{O}_S(l))<d,\end{eqnarray}
\begin{eqnarray}\label{eqineqassump2dp2}1+(l+1)^2+l+1-d=h^0(S,\mathcal{O}_S(l+1))-d\geq 2s.
\end{eqnarray}
Then, for a general subscheme
$Z_d\subset S$ of length $d$, a general vector bundle $E$ constructed from an extension class $e\in {\rm Ext}^1(\mathcal{I}_{Z_d}(d+1),\mathcal{O}_S)$, and general set of $s$ points $x_1,\,\ldots,\,x_s\in S$, there exists a section $\sigma'$ of $E$ vanishing at all points  $x_i$, and whose zero-set is of dimension $0$.
\end{lemm}
\begin{proof} Mutatis mutandis, the proof is the same as that of Lemma \ref{lecasgeneric}.
\end{proof}
Proposition \ref{prokeypropdp2} is now proved.
\end{proof}

Let us now  prove Theorem \ref{theo2surfdelPsanspoint}(a), which is the following statement.
\begin{theo} \label{theogrosseborne} Let $S$ be a smooth del Pezzo surface of degree $2$ over a field $K$ of characteristic $0$. Then any effective $0$-cycle $z\in{\rm CH}_0(S)$ can be written as
\begin{eqnarray}\label{eqpourfindp2suppl}
 z=z'+\gamma h_2\,\,{\rm in}\,\,{\rm CH}_0(S),
\end{eqnarray}
where $z'$ is effective of degree $\leq 13$.
\end{theo}
\begin{proof} Let $z\in {\rm CH}_0(S)$ be an effective $0$-cycle of degree $d$ and let $l$ be such that
\begin{eqnarray}\label{eqcrucialpourvbdp2} h^0(S,\mathcal{O}_S(l))=1+l^2+l< d.\end{eqnarray}
Assume first that $d\leq h^0(S,\mathcal{O}_S(l+1))-2$. Then $(l+1)^2h_2-z$ is effective thanks to Lemma \ref{leeffective}, which applies since $l\geq 1$ and $\mathcal{O}_S(2)$ is very ample. Replacing $z$ by $(l+1)^2h_2-z$ if necessary, we can thus assume
$d\leq 2(l+1)^2-d$.
 Thus,  from now on, we can assume that
\begin{eqnarray}\label{eqdlplusunsurdeux} d\leq (l+1)^2.
\end{eqnarray}
Using (\ref{eqdlplusunsurdeux}), we get that
$$h^0(S,\mathcal{O}_S(l+1))-d\geq l+2$$
so once $l\geq 2$, we get that $h^0(S,\mathcal{O}_S(l+1))-d\geq 4$. We can
thus apply Proposition \ref{prokeypropdp2} and thus we proved  that $z-h_2$ is effective when ${\rm deg}\,z\geq 1+h^0(S,\mathcal{O}_S(2))=8$, unless we are in one of the following cases :
\begin{eqnarray}\label{eqmauvaistroplplusdeux}{\rm deg}\,z=h^0(S,\mathcal{O}_S(l+1)),\,\,{\rm or}\,\, {\rm deg}\,z=h^0(S,\mathcal{O}_S(l+1))-1.\end{eqnarray}
We now treat  the  missing cases  (\ref{eqmauvaistroplplusdeux}). Let $z':=z+h_2$ and $d':={\rm deg}\,z'=d+2$. Then
$$d'=h^0(S,\mathcal{O}_S(l+1))+2,\,\,{\rm  resp.}\,\,d'=h^0(S,\mathcal{O}_S(l+1))+1.$$ Furthermore
$$h^0(S,\mathcal{O}_S(l+2))-d'=2(l+2)-2,\,\,
{\rm resp.}\,\, h^0(S,\mathcal{O}_S(l+2))-d'=2(l+2)-1.$$
We can thus  apply  Proposition \ref{prokeypropdp2} to $z'$ with $s=4$ once $2(l+2)-2\geq 8$, that is $l\geq 3$.  So we proved that, in the  cases  (\ref{eqmauvaistroplplusdeux}),
$$z'-2h_2=z-h_2$$ is effective if $l\geq 3$. When $l=2$ and (\ref{eqmauvaistroplplusdeux}) holds, we have $d=13$ or $d=12$.
Putting everything together, we have proved  by induction that any effective cycle $z\in {\rm CH}_0(S)$ can be written
as
\begin{eqnarray}\label{eqpresquedp2fin} z=\pm z'+\gamma h_2\,\,{\rm in}\,\,{\rm CH}_0(S), \end{eqnarray} where $z'$ is effective of degree $\leq 13$. However, there is an involution of $S$ which follows from the existence of the  degree $2$  morphism  $S\rightarrow \mathbb{P}^2$ given by the anticanonical system. This  involution $\iota$ has the property that for any $z\in {\rm CH}_0(S)$,
\begin{eqnarray}\label{eqtructrucfin} z+\iota(z)=({\rm deg}\,z)h_2
\,\,{\rm in}\,\,{\rm CH}_0(S).
\end{eqnarray}
Using (\ref{eqtructrucfin}), we can arrange to make the sign in (\ref{eqpresquedp2fin}) positive, so
Theorem \ref{theogrosseborne} is proved.
\end{proof}

We finally prove Theorem \ref{theo2surfdelPsanspoint}(b).
\begin{proof}[Proof of Theorem \ref{theo2surfdelPsanspoint}(b)] By Theorem \ref{theogrosseborne}, it suffices to study the case of  an effective $0$-cycle $z$ of degree $d\leq 11$. When  $10\leq d\leq11$, we argue as follows. In this case,
$$d\leq h^0(S,\mathcal{O}_S(3))-2$$
so that by  Lemma \ref{leeffective}
$$ z':=9h_2-z_d$$
is effective, with ${\rm deg}\,z'<{\rm deg}\,z$. We are thus reduced to the case where $d=8,\,9$. As $h^0(S,\mathcal{O}_S(2))=7$, we can apply in these cases the vector bundle method of Section \ref{secvb} with $l=2$. As we have $h^0(S,\mathcal{O}_S(3))=13$, we get in both cases $h^0(S,\mathcal{O}_S(3))-d\geq 4$, hence Proposition \ref{prokeypropdp2} tells that in both cases, $z-h_2$ is effective of degree $\leq 7$.
\end{proof}

\section{Zero-cycles on del Pezzo surfaces of degree $1$ \label{secdelpezzodeg1}}
In the case of a del Pezzo surface $S$ of degree $d_S=1$, using again the notation $-K_S=:\mathcal{O}_S(1)$, we have
$$h^0(S,\mathcal{O}_S(l))=1+\frac{1}{2}(l^2+l)$$
for $l\geq 0$. Note that $S$ has a point, defined as the base-locus of the linear system $|-K_S|$. As usual we will denote its class by $h_1\in{\rm CH}_0(S)$.
We note that in this case the line bundle $\mathcal{O}_S(3)$ is  very ample, as one sees by looking at its restrictions to the elliptic curves in the pencil $|\mathcal{O}_S(1)|$. The line bundle $\mathcal{O}_S(2)$ is generated by sections and induces a degree $2$ morphism onto a surface of degree $2$  in $\mathbb{P}^3$, which has to be a singular quadric.

The analog of Proposition \ref{prokeyprop} is now
\begin{prop}\label{prokeypropdp1} Let $S$ be a del Pezzo surface of degree $1$ defined over a field $K$ of characteristic $0$. Let $d,\,l\geq 1,\,s$ be three integers such that
\begin{eqnarray}\label{eq1dp1} h^0(S,\mathcal{O}_S(l))=1+\frac{1}{2}(l^2+l) <d\\
\label{eq2dp2}  h^0(S,\mathcal{O}_S(l+1))\geq d+2s.
\end{eqnarray}
Then for any  effective $0$-cycle $z_d\in {\rm CH}_0(S)$, the cycle $z_d-sh_1$ is effective.
\end{prop}
 The proof is similar to the proof of  Proposition \ref{prokeyprop} and we will not repeat the details here.

We will now  prove
 Theorem \ref{theo2surfdelPezzodeg1}(b), which is the following statement.
\begin{theo} \label{theo2surfdelPezzodeg1variant}   Let   $S$ be a  smooth degree-$1$ del Pezzo surface over  a field $K$ of characteristic $0$. Then  any  effective $0$-cycle $z\in{\rm CH}_0(S)$ can be written as
$$ z=  \pm z'+\gamma h_1,
 $$
 where $z'$ is effective of degree $15,\,10,\,7,\,6$ or $\leq 4$.
\end{theo}
\begin{proof}
Let $z_d\in {\rm CH}_0(S)$ be an effective $0$-cycle of degree $d\geq 5$. We choose $l$ such that
$$h^0(S,\mathcal{O}_S(l))<d\leq  h^0(S,\mathcal{O}_S(l+1)).$$
Note that, in particular $l\geq 2$, hence $\mathcal{O}_S(l+1)$ is very ample.

If $d\leq h^0(S,\mathcal{O}_S(l+1))-2$, we get by Lemma \ref{leeffective} that $(l+1)^2h_1-z_d$ is effective, so we can assume up to replacing $z_d$ by  $(l+1)^2h_1-z_d$, that
\begin{eqnarray}\label{eqineqpourdp1first} d\leq \frac{1}{2}(l+1)^2.\end{eqnarray}
It follows from (\ref{eqineqpourdp1first}) that
$$ h^0(S,\mathcal{O}_S(l+1))-d\geq 1+\frac{1}{2}(l+1).$$
As $l\geq 2$,  $1+\frac{1}{2}(l+1)\geq 2$, hence Proposition \ref{prokeypropdp1} applies with $s=1$ and tells that $z_d-h_1$ is effective.

We next have to consider the cases where
\begin{enumerate}
\item \label{i} $d= h^0(S,\mathcal{O}_S(l+1))$,
\item  \label{ii}  $d= h^0(S,\mathcal{O}_S(l+1))-1=\frac{(l+1)(l+2)}{2}$.
\end{enumerate}
In  Case \ref{i}, we replace $z_d$ by $z'=z_d+h_1$, and in Case \ref{ii}, we replace $z_d$ by $z'=z_d+2h_1$. In both cases, we have ${\rm deg}\,z'=h^0(S,\mathcal{O}_S(l+1))+1$.

Furthermore, we have
$$h^0(S,\mathcal{O}_S(l+2))-{\rm deg}\,z'=l+1.$$

In  Case \ref{i}, if $l\geq 3$, we get by Proposition \ref{prokeypropdp1} that
$z'-2h_1$ is effective, hence $z_d-h_1$ is effective.

In Case \ref{ii}, if  $l\geq 5$, , we get by Proposition \ref{prokeypropdp1} that
$z'-3h_1$ is effective, hence $z_d-h_1$ is effective.

In conclusion, when $d\geq 5$, the only cases where we cannot conclude that $\pm z_d-h_1$ is effective of degree strictly smaller than $d$ is when we are in Case \ref{ii}  with $2\leq l\leq 4$ or  in Case \ref{i} with $l= 2$.
 Case \ref{ii}  with $l= 4,\,3,\,2$ provides respectively $d=15,\,10,\,6$. Case \ref{i} with $l= 2$ provides $d=7$. So Theorem \ref{theo2surfdelPezzodeg1}(b) is proved.
\end{proof}
\begin{proof}[Proof of Theorem \ref{theo2surfdelPezzodeg1}(a)] In view of Theorem \ref{theo2surfdelPezzodeg1variant}, it suffices to prove the following
\begin{lemm}\label{lefindefin} Let $S$ be a del Pezzo surface of degree $1$ defined over a field $K$ of characteristic $0$. Then for any effective cycle $z$ of degree $d\leq 15$, we can write
$$z=\lambda h_1-z'\,\,{\rm in}\,\in{\rm CH}_0(S),$$
where $\lambda\in\mathbb{Z}$ and $z'$ is effective of degree $\leq 15$.
\end{lemm}
\begin{proof} If ${\rm deg}\,z=15$, using the fact that $h^0(S,\mathcal{O}_S(5))=16$ and Lemma \ref{levariant}, we find that $z':=30 h_1-z$ is effective. As ${\rm deg}\,z'=15$, the lemma is proved in this case. If $10\leq {\rm deg}\,z\leq 14$, Lemma \ref{leeffective} tells that $z':=25 h_1-z$ is effective. As ${\rm deg}\,z'\leq 15$, the lemma is proved in this case. If $1\leq {\rm deg}\,z\leq 9$, Lemma \ref{leeffective} tells that $z':=16 h_1-z$ is effective, hence the lemma is fully proved.
\end{proof}
Theorem \ref{theo2surfdelPezzodeg1}(a) is now proved.
\end{proof}


\begin{thebibliography}{99}
\bibitem{bogothschi} F. A. Bogomolov, Yu. Tschinkel, Density of rational points on elliptic K3 surfaces.
Asian J. Math. 4 (2000), no. 2, 351-368.
\bibitem{cheah}  J. Cheah, Cellular decompositions for nested Hilbert schemes of points, Pacific. J. Math. 183, No.1 (1998), 39-90.
\bibitem{colliot}  J.-L. Colliot-Th\'{e}l\`{e}ne,  Z\'{e}ro-cycles sur les surfaces de del Pezzo (Variations sur un th\`{e}me de Daniel Coray).  Enseign. Math. 66 (2020), no. 3-4, 447-487.
\bibitem{coray} D. F. Coray. Algebraic points on cubic hypersurfaces, Acta Arith. 30 (1976), no. 3,
267-296.

\bibitem{SEggay}   P. Engel, O. de Gaay Fortman, S. Schreieder. Matroids and the integral Hodge conjecture for abelian varieties, arXiv:2507.15704.
    \bibitem{fogarty} J. Fogarty. Algebraic families on an algebraic surface,
Amer. J. Math. 90 (1968), 511-521.
    \bibitem{fulton} W. Fulton. Intersection theory. Ergebnisse der Mathematik und ihrer Grenzgebiete (3)
[Results in Mathematics and Related Areas (3)], 2. Springer-Verlag, Berlin, 1984.
\bibitem{griffithsharris} P. Griffiths, J. Harris. {\it Principles of algebraic geometry}. New York: Wiley (1978).
\bibitem{kollarmella} J. Koll\'ar, M. Mella. Quadratic families of elliptic curves and unirationality of
degree $1$ conic bundles. Amer. J. Math. 139 (2017), no. 4, 915-936.
\bibitem{lazarsfeld} R. Lazarsfeld. {\it Positivity in algebraic geometry. II. Positivity for vector bundles, and multiplier ideals}. Ergebnisse der Mathematik und ihrer Grenzgebiete. 3. Folge. A Series of Modern Surveys in Mathematics  49. Springer-Verlag, Berlin, 2004.
\bibitem{maq} Qixiao Ma. Closed points on cubic hypersurfaces, Michigan Math. J. 70 (2021), no. 4, 857-868.
\bibitem{schwarzenberger} R. L. E. Schwarzenberger.
Vector bundles on the projective plane.
Proc. London Math. Soc. (3) 11 (1961), 623-640.
\bibitem{voisinjems} C. Voisin. On the universal ${\rm CH}_0$ group of cubic hypersurfaces. J. Eur. Math. Soc.  19 (2017), no. 6, 1619-1653.
\end{thebibliography}
\end{document}